\documentclass[12pt, a4paper]{amsart}

\calclayout

\usepackage[english]{babel}
\usepackage[utf8]{inputenc}
\usepackage{amsmath}
\usepackage{graphicx}
\usepackage[colorinlistoftodos]{todonotes}
\usepackage{amssymb,amscd,latexsym,epsfig,xy}
\usepackage[mathscr]{euscript}
\usepackage{amsthm}
\usepackage{tikz}
\usepackage{MnSymbol}
\usetikzlibrary{matrix,arrows,decorations.pathmorphing}
\usepackage{hyperref}
\usepackage{enumerate}
\usepackage{float}
\usepackage{tikz-cd}
\usepackage{relsize}
\usepackage{mathtools}

\newtheorem{thm}{Theorem}[section]
\newtheorem{lem}{Lemma}[section]
\newtheorem{cor}{Corollary}[section]

\newtheorem{prop}{Proposition}[section]

\newcommand\A{\mathbb A}

\newcommand\F{\mathbb F}

\newcommand\NN{\mathbb N}

\newcommand\PP{\mathbb P}
\newcommand\R{\mathbb R}
\newcommand\Q{\mathbb Q}
\newcommand\Z{\mathbb Z}

\newcommand\cC{\mathcal C}

\newcommand\cO{\mathcal O}

\newcommand\cV{\mathcal V}

\newcommand\fm{\mathfrak m}
\newcommand\bd{\mathbf d}
\newcommand\Aut{\operatorname{Aut}}

\newcommand\Hom{\operatorname{Hom}}

\newcommand\ev{\operatorname{ev}}

\newcommand\trop{\mathrm{trop}}
\newcommand\virt{\mathrm{virt}}
\newcommand\pt{\mathrm{pt}}

\title[On an Example of Quiver DT/Relative GW correspondence]{On an Example of Quiver Donaldson-Thomas/Relative Gromov-Witten correspondence}

\author{Pierrick Bousseau}

\date{}

\begin{document}

\begin{abstract}
We explain and generalize a recent result of Reineke-Weist by showing how to reduce it to the 
Gromov-Witten/Kronecker
correspondence by a degeneration and blow-up. We also refine the result by working with all genera on the Gromov-Witten side and with refined Donaldson-Thomas invariants on the quiver side.
\end{abstract}

\maketitle

\setcounter{tocdepth}{1}
\renewcommand\contentsname{\vspace{-0mm}}

\tableofcontents

\thispagestyle{empty}

\section{Introduction}

\subsection{Statement of the result}

Let $Y$ be a smooth projective complex surface and let $D_1$ and $D_2$ be two smooth non-empty divisors on $Y$, 
intersecting transversally, and such that the 
union $D_1 \cup D_2$ is anticanonical. 
Elementary  theory of surfaces implies that $D_1$ and $D_2$
are necessarily rational curves and that the intersection $D_1 \cap D_2$ consists of two points. An example to keep in mind is 
$Y=\PP^2$, $D_1$ a line and $D_2$ a smooth conic not tangent to $D_1$.

For every $\beta \in H_2(Y,\Z)$ such that
$\beta \cdot D_1>0$ and
$\beta \cdot D_2>0$, let 
$N_{0,\beta}^{Y/D_1}$ be the Gromov-Witten count of rational curves in $Y$ of class $\beta$, intersecting $D_1$ 
in a unique point with maximal tangency and passing through $\beta \cdot D_2$ prescribed points in general position in $Y$.
In the present paper, we explain how to construct a quiver $Q^{Y/D_1}_\beta$ and a dimension vector 
\[\mathbf{d}(\beta) \in \Z^{(Q^{Y/D_1}_\beta)_0}\,,\] 
where $(Q^{Y/D_1}_\beta)_0$
is the set of vertices of
$Q^{Y/D_1}_\beta$.
Donaldson-Thomas theory for quivers defines an integer $\Omega_{\beta}^{Y/D_1}$, 
virtual count of representations of $Q^{Y/D_1}_\beta$
of dimension vector $\mathbf{d}(\beta)$.
A priori, $\Omega_\beta^{Y/D_1}$ is only defined if
$\mathbf{d}(\beta) \in \NN^{(Q^{Y/D_1}_\beta)_0}$.
If $\mathbf{d}(\beta) \notin \NN^{(Q^{Y/D_1}_\beta)_0}$,
we set $\Omega_\beta^{Y/D_1} \coloneq 0$.

Our main result is the following theorem.

\begin{thm} \label{main_thm_0}
If the quiver $Q^{Y/D_1}_\beta$ is acyclic,
that is, does not contain 
any oriented cycle,
then we have 
\[  \Omega_{\beta}^{Y/D_1}
=(-1)^{\beta \cdot D_1+1}  N_{0,\beta}^{Y/D_1}\,.\]
\end{thm}

In the previous result, $D_1$ and $D_2$ do not play symmetric roles. 
Indeed, we are counting curves with a tangency condition along $D_1$ whereas
$D_2$ only appears in an indirect way in the number $\beta \cdot D_2$ of point conditions. 
If we exchange the roles of $D_1$ and $D_2$, we get a new example.

\textbf{Example:} Let $Y=\PP^2$, $D_1$ a line and 
$D_2$ a smooth conic not tangent to $D_1$. Let 
$d$ be a positive integer. Taking 
$\beta=d \in \Z=H_2(\PP^2,\Z)$, the quiver 
$Q^{Y/D_1}_\beta$ consists of 
$2d+1$ vertices, 
$i_1, \dots, i_{2d}$ and $j$, and $2d$ arrows 
$\alpha_k \colon i_k \rightarrow j$ for all 
$k=1,\dots, 2d$. The dimension vector 
$\mathbf{d}(\beta)$ is $1$ on the vertices 
$i_1,\dots,i_{2d}$, and $d$ on the vertex $j$.
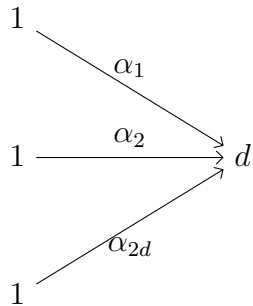
\begin{center}
\begin{tikzpicture}[>=angle 90]
\matrix(a)[matrix of math nodes,
row sep=3em, column sep=3em,
text height=1.5ex, text depth=0.25ex]
{1& &\\
1& &d\\
1& &\\};
\path[->](a-1-1) edge node[above]{$\alpha_1$} (a-2-3);
\path[->](a-2-1) edge node[above]{$\alpha_2$} (a-2-3);
\path[->](a-3-1) edge  node[below]{$\alpha_{2d}$} (a-2-3);
\end{tikzpicture}
\end{center}
This quiver is acyclic and so Theorem \ref{main_thm_0} applies. 
In this case, Theorem \ref{main_thm_0} reduces to the main result obtained by
\cite{reineke2018moduli}. More generally, 
Theorem \ref{main_thm_0} gives a way to produce a large number of new examples, 
see Section \ref{section_examples}.

We will in fact prove a more general version of Theorem \ref{main_thm_0}, 
involving higher genus Gromov-Witten invariants and 
refined Donaldson-Thomas invariants of the quiver
$Q_{\beta}^{Y/D_1}$. The moduli space of genus $g$ stable maps in $Y$ of class $\beta$, intersecting $D_1$ in a unique point with maximal tangency and passing through 
$\beta \cdot D_2$ prescribed points, has virtual dimension $g$. We define Gromov-Witten invariants $N_{g,\beta}^{Y/D_1}$ by integrating $(-1)^g \lambda_g$ over the moduli space, where $\lambda_g$ is the top Chern class of the Hodge bundle. On the other hand, refined Donaldson-Thomas theory for quivers defines a Laurent polynomial $\Omega_\beta^{Y/D_1}(q^{\frac{1}{2}}) \in \Z[q^{\pm \frac{1}{2}}]$, refined virtual count of representations of $Q_\beta^{Y/D_1}$ of dimension vector $\mathbf{d}(\beta)$.

\begin{thm}\label{main_thm_precise_intro}
If the quiver $Q_{\beta}^{Y/D_1}$ is acyclic, 
then we have an equality of formal power series in $\hbar$:
\[ \Omega_\beta^{Y/D_1}(q^{\frac{1}{2}})
=(-1)^{\beta \cdot D_1+1} 
\left( 2 \sin \left( \frac{\hbar}{2}
\right) \right)
\left( 
\frac{\hbar}{2 \sin \left( \frac{\hbar}{2} \right)}
\right)^{\beta \cdot D_2}
\left( \sum_{g\geqslant 0} N_{g,\beta}^{Y/D_1} 
\hbar^{2g-1} \right) \,,\]
where $q=e^{i \hbar}=\sum_{n \geqslant 0} \frac{(i\hbar)^n}{n!}$.
\end{thm}

We recover Theorem \ref{main_thm_0} from
Theorem \ref{main_thm_precise_intro} in the limit $\hbar \rightarrow 0$, 
$q^{\frac{1}{2}} \rightarrow 1$.
Theorem \ref{main_thm_precise_intro} is new even in the simplest case of 
$Y=\PP^2$, $D_1$ a line and $D_2$ a smooth conic not tangent to $D_1$.

\subsection{Origin of the result}
The present paper comes from a search for a
``natural explanation" of the result of \cite{reineke2018moduli}. 
The source of our explanation is the correspondence existing between log Gromov-Witten 
invariants of log Calabi-Yau surfaces and quiver Donaldson-Thomas invariants. 
This known correspondence comes from the combination of the work of
\cite{MR2650811} on the quiver side and of the work of
\cite{MR2667135} on the Gromov-Witten side. Examples of this correspondence are discussed in 
\cite{MR2662867}
and in 
\cite{MR3004575}
under the name of the Gromov-Witten/Kronecker correspondence. 
The most general version of this correspondence, that we will use in this paper, can be found in Section 8.5 of
\cite{bousseau2018quantum_tropical}. 
In fact, in Section 8.5 of \cite{bousseau2018quantum_tropical}, a generalization of the original genus 0/DT correspondence is given, involving higher genus Gromov-Witten invariants with lamba class insertion and refined DT invariants.

Reineke and Weist observed the close similarity between their result and the Gromov-Witten/Kronecker correspondence 
but remarked that a direct connection does not seem obvious. The core of the present paper is to explain such connection.

The input of the general form of the Gromov-Witten/Kronecker correspondence described in Section 8.5 of \cite{bousseau2018quantum_tropical} 
is a log Calabi-Yau surface, that is, the pair of a smooth projective surface $Z$ and of an anticanonical divisor $D$ in $Z$. 
We then consider Gromov-Witten counts of rational curves in $Z$ in a fixed curve class $\beta_Z$,
intersecting $D$ in only one point with maximal tangency.
This seems different from our $N_{0,\beta}^{Y/D_1}$, counting rational curves in $Y$ of class $\beta$, 
intersecting $D_1$ in a unique point and passing through $\beta \cdot D_2$ points in general position in $Y$.
The key geometric idea of the present paper is the following: in order to go from $Y$ to some $Z$, 
we move the $\beta \cdot D_2$ points onto $D_2$ and we blow them up.

So, let $Z$ be the surface obtained from 
$Y$ by blowing-up $\beta \cdot D_2$ points on $D_2$, distinct from each other and distinct from 
$D_1 \cap D_2$. Denote 
$\pi_Y \colon Z \rightarrow Y$ the blow-up morphism, 
$F_1, \dots, F_{\beta \cdot D_2}$ the exceptional divisors and $\beta_Z
\coloneqq \pi^{*}_Y \beta -\sum_{j=1}^{\beta \cdot D_2}
[F_j] \in H_2(Z,\Z)$. 
We still denote $D_1$ and $D_2$ the strict transforms of $D_1$ and $D_2$ in $Z$.
Let $N_{0,\beta_Z}^{Z/D_1}$ be the log 
Gromov-Witten count of rational 
curves in $Z$ of class $\beta_Z$, intersecting $D_1$ in a unique point with 
maximal tangency and not intersecting $D_2$.

The intuitive picture of moving the $\beta \cdot D_2$ points from a general configuration in $Y$ 
to a special configuration on $D_2$ suggests the following result.

\begin{thm} \label{thm_gw_0}
We have $N_{0,\beta}^{Y/D_1}
=N_{0,\beta_Z}^{Z/D_1}$.
\end{thm} 

In absolute Gromov-Witten theory, trading exceptional divisors for point insertions is quite familiar, see \cite{MR1832328} for some early example. The main novelty in Theorem \ref{thm_gw_0} is that we are dealing with relative Gromov-Witten theory.

Theorem \ref{main_thm_0} follows
from Theorem \ref{thm_gw_0} in a natural fashion. Indeed,
the invariants $N_{0,\beta_Z}^{Z/D_1}$ are exactly those entering in the general form of the Gromov-Witten/Kronecker 
correspondence and a quiver can be obtained from a toric model of the log Calabi-Yau surface $(Z, D_1 \cup D_2)$. 
We review the construction of such a quiver 
in Section
\ref{quiver_construction}. In the particular case of $Y=\PP^2$, 
$D_1$ a line and $D_2$ a smooth conic not tangent to 
$D_1$, we recover the quiver considered by Reineke and Weist. A feature of our approach is that we do not have to guess the quiver, which
contrasts with the situation in \cite{reineke2018moduli}.
In particular, we can produce many new examples, some of them being presented in Section
\ref{section_examples}.

\subsection{Technical content}
The technical content of this paper is the proof of Theorem \ref{thm_gw_0} and consists in two degeneration arguments in Gromov-Witten theory.
The 1st step involves degenerating the 
$\beta \cdot D_2$ points on $D_2$. The 2nd step involves blowing these points up. 
Because we will be really considering curves with tangency conditions along $D_1 \cup D_2$, which is a singular normal crossing divisor, 
we will be working with log Gromov-Witten invariants.
In this context, our degeneration arguments will use the decomposition formula of
\cite{abramovich2017decomposition}
in a way similar to how we used it
in \cite{bousseau2017tropical} and
\cite{bousseau2018quantum_tropical}, and then an explicit treatment of the gluing.
In particular, our setting is explicit enough so that a general theory of gluing for log Gromov-Witten invariants, such as recently developed in 
\cite{ranganathan2019logarithmic} or
\cite{abramovich2019punctured}, is not necessary. 

Let us stress the most technical point, which is not present in \cite{bousseau2017tropical} or
\cite{bousseau2018quantum_tropical}. 
In all cases, the general decomposition formula requires to identify rigid tropical curves, indexing the terms contributing in the formula.
In \cite{bousseau2017tropical} and
\cite{bousseau2018quantum_tropical}, the components of the special fibers of the degenerations are toric, or close enough of toric, so that the tropical balancing condition gives enough constraints to identify exactly the relevant contributing rigid tropical curves. In the present paper, the degenerations considered are far enough from being toric and purely combinatorial arguments are not enough to rule out rigid tropical curves which eventually will contribute zero in the degeneration formula. To eliminate those, we have to argue at the level of stable log maps. The issue is that rigid tropical curves are not necessarily realizable as tropicalization of stable log maps (see Example 6.4 of \cite{abramovich2017decomposition} for an example). 
The solution to this problem is to study the possible tropicalizations of all the
relevant stable log maps, even the ones giving non-rigid tropical curves, and then 
to identify the contributing rigid tropical curves as possible rigid limits of them.

\subsection{Relation with other works}
We mention briefly two works that are logically independent of the present paper but 
which consider some of the same objects from a different point of view.

If $D_1$ is nef, it follows from the main result of \cite{van2017local}
that the relative genus zero Gromov-Witten invariants $N_{0,\beta}^{Y/D_1}$
are related to the absolute genus zero Gromov-Witten invariants
    with $\beta \cdot D_2$ point insertions of the non-compact three-fold total space of the line bundle $\cO_Y(-D_1)$. Thus, if $D_1$ is nef, we can view Theorem \ref{main_thm_0} as establishing a relation between absolute genus $0$ Gromov-Witten theory of the three-fold $\cO_Y(-D_1)$ and Donaldson-Thomas invariants of the quiver $Q_\beta^{Y/D_1}$.

For $Y=\PP^2$, $D_1$ a line and $D_2$ a smooth conic not tangent to $D_1$,
     \cite{reineke2018moduli} use the quiver interpretation of $N_\beta^{Y/D_1}$ to derive a recursion relation satisfied by the invariants $N_\beta^{Y/D_1}$. In work 
     \cite{fan2018wdvv}, this recursion is derived as a very particular case of a general WDVV formalism for genus zero Gromov-Witten invariants relative to a smooth divisor. It would be interesting to study, either from the quiver or Gromov-Witten sides, if similar recursion relations exist for the examples presented in Section \ref{section_examples}.

\subsection{Notations} In order to keep the amount of notations as low as possible,
we systematically use the convention to denote in the same way a divisor and its strict transform by some blow-up morphism, hoping that the context will be clear enough to
specify the relevant geometry. 
We denote $A_*$ and $A^*$ for respectively Chow groups and operatorial 
cohomology Chow groups, as in \cite{MR1644323}.

\subsection{Plan of the paper}
In Section
\ref{section_precise_statements}, we state Theorem \ref{main_thm_precise}, cited as Theorem \ref{main_thm_precise_intro}
in the Introduction. 
The following two sections are dedicated to the 
proof of Theorem \ref{main_thm_precise}. 
In Section
\ref{section_deg_point}, 
we apply a degeneration argument to show that we can 
transform the $\beta \cdot D_2$ absolute point conditions entering the definition of 
$N_\beta^{Y/D_1}$ into $\beta \cdot D_2$ relative point conditions along the divisor $D_2$.
In Section \ref{section_exchange_blow_ups}, another degeneration argument shows that these $\beta \cdot D_2$ relative point conditions along 
$D_2$ can be transformed into no condition on the surface $Z$ obtained from 
$Y$ by blowing-up $\beta \cdot D_2$ points on $D_2$.
In Section \ref{section_end_proof}, we  finish the proof of Theorem \ref{main_thm_precise}. 
In Section \ref{section_examples}, we present various examples.

\subsection{Note (v2).} The first arxiv version und unfortunately the published version of this paper contain incorrect normalization factors in the statement of Theorems 1.2, 2.1 and of Proposition 3.1, the mistake in the proof being in the very last part of Section 3. The statements and the proof are corrected in the present version. I thank Longting Wu for asking me a question leading to the realization of the need of the correction. 

\subsection{Acknowledgements.}
I made the key observation allowing this paper to exist (the fact that for $Y=\PP^2$, $D_1$ a line and $D_2$ a conic, the quiver associated to the surface $Z$ by the recipe of \cite{bousseau2018quantum_tropical} coincides with the quiver of 
\cite{reineke2018moduli}) in a TGV Lyria train between Paris and Zurich. I thank Honglu Fan, Rahul Pandharipande, Richard Thomas and Longting Wu for  useful discussions. 

I acknowledge the support of Dr.\ Max R\"ossler, the Walter Haefner Foundation and the ETH Z\"urich
Foundation.

\section{Main result}\label{section_precise_statements}

\subsection{The Gromov-Witten side} \label{gw_side}

Let $Y$ be a smooth projective complex surface and let $D_1$ and $D_2$ be two smooth non-empty divisors on $Y$, intersecting transversally, and such that the 
union $D_1 \cup D_2$ is anticanonical. Elementary  theory of surfaces implies that $D_1$ and $D_2$
are necessarily rational curves and that the intersection $D_1 \cap D_2$ consists of two points. 

We fix $\beta \in H_2(Y,\Z)$ such that 
$\beta \cdot D_1>0$ and 
$\beta \cdot D_2>0$.
Let 
$\overline{M}_{g, \beta \cdot D_2}(Y/D_1,\beta)$
be the moduli space of $(\beta \cdot D_2)$-pointed genus $g$ class $\beta$ 
stable maps to $Y$ relative to $D_1$, with contact condition along $D_1$ at a single point and so with contact order $\beta \cdot D_1$.
When we say $(\beta \cdot D_2)$-pointed, we refer  to marked points
in addition to the marked point keeping track of the tangency condition along $D_1$.
The moduli space
$\overline{M}_{g, \beta \cdot D_2}(Y/D_1,\beta)$
is a proper 
Deligne-Mumford stack of virtual dimension 
\[(1-g)(\dim Y-3)+\beta \cdot c_1(Y)+\beta \cdot D_2-(\beta \cdot D_1-1)=g+2(\beta \cdot D_2)\,,\]
where we used that $D_1 \cup D_2$ is anticanonical and so 
$\beta \cdot c_1(Y)=\beta \cdot D_1+\beta \cdot D_2$.
It admits a virtual fundamental class 
\[[\overline{M}_{g, \beta \cdot D_2}(Y/D_1,\beta)]^\virt \in A_{g+2(\beta \cdot  D_2)}(
\overline{M}_{g, \beta \cdot D_2}(Y/D_1,\beta),\Q) \,.\]
Let $\ev_k \colon \overline{M}_{g, \beta \cdot D_2}(Y/D_1,\beta)
\rightarrow Y$, for $1 \leqslant k \leqslant \beta \cdot D_2$, be the evaluation maps at the 
$\beta \cdot D_2$ marked points. 
We define 
\[N_{g,\beta}^{Y/D_1} \coloneqq 
\int_{[\overline{M}_{g, \beta \cdot D_2}(Y/D_1,\beta)]^\virt}
(-1)^g \lambda_g \prod_{k=1}^{\beta \cdot D_2} \ev_k^{*}(pt) \in \Q\,,\]
where $\pt \in A^1(Y)$ is the class of a point and $\lambda_g$ is the top Chern class of the Hodge bundle over $\overline{M}_{g, \beta \cdot D_2}(Y/D_1,\beta)$.
Recall that if 
\[\nu \colon \cC \rightarrow \overline{M}_{g, \beta \cdot D_2}(Y/D_1,\beta)\]
denotes the universal source curve, then the Hodge bundle is the rank $g$ vector bundle 
$\nu_{*} \omega_{\nu}$, where $\omega_\nu$ is the dualizing line bundle relative to 
$\nu$.

\subsection{Basics on log Calabi-Yau surfaces} \label{basics_log_cy}
To explain the construction of the quiver 
$Q_\beta^{Y/D_1}$ given in the next Section
\ref{quiver_construction}, it is useful to recall some terminology related to 
log Calabi-Yau surfaces. We refer to \cite{friedman2015geometry},
\cite{MR3314827}, \cite{MR3415066} for more details.
A log Calabi-Yau surface with maximal boundary is a pair 
$(Y,D)$ with $Y$ a smooth projective surface and $D$ a singular reduced 
normal crossing anticanonical divisor of $Y$.
There are two basic operations on log Calabi-Yau surfaces with maximal boundary:
\begin{itemize}
    \item Corner blow-up. If $(Y,D)$ is a log Calabi-Yau surface with 
    maximal boundary, then $(\tilde{Y}, \tilde{D})$ is a log Calabi-Yau surface with maximal boundary, 
    where $\tilde{Y}$ is a blow-up of $Y$ at one singular point of $D$ and 
    $\tilde{D}$ is the preimage of $D$ in $\tilde{Y}$.
    \item Interior blow-up. If $(Y,D)$ is a log Calabi-Yau surface with maximal boundary, 
    then $(\tilde{Y},\tilde{D})$ is a log Calabi-Yau surface with 
    maximal boundary, where $\tilde{Y}$ is a blow-up of $Y$ at some smooth 
    points of $D$ (with infinitely near points allowed) and $\tilde{D}$ is the strict transform of $D$ in $\tilde{Y}$.
\end{itemize}
The simplest example of log Calabi-Yau surface with maximal boundary is 
$(\bar{Y},\bar{D})$, with $\bar{Y}$ a smooth projective toric surface and 
$\bar{D}$ the union of toric divisors of $\bar{Y}$. Such log Calabi-Yau surface with maximal boundary is called toric.

A toric model of a log Calabi-Yau surface with maximal boundary $(Y,D)$ 
is a morphism $\pi \colon (Y,D) \rightarrow (\bar{Y},\bar{D})$, where $(\bar{Y},\bar{D})$
is a toric log Calabi-Yau surface with maximal boundary, and where $\pi$ is obtained from $(\bar{Y},\bar{D})$ by a sequence of interior blow-ups.

We recall Proposition 1.3 of \cite{MR3415066}:

\begin{prop} \label{prop_toric_model}
For every $(Y,D)$ log Calabi-Yau surface with maximal boundary, there exists $(\tilde{Y},\tilde{D})$, obtained from 
$(Y,D)$ by a sequence of corner blow-ups, and admitting a toric model
$(\tilde{Y},\tilde{D}) \rightarrow (\bar{Y},\bar{D})$
\end{prop}

 Given $(Y,D)$, there are in general many (often infinitely many)
$(\tilde{Y},\tilde{D})$ and $(\bar{Y},\bar{D})$ as in Proposition
\ref{prop_toric_model}. The non-uniqueness of toric models is related 
to the theory of cluster mutations, see \cite{MR3350154}.

\subsection{Quiver construction} \label{quiver_construction}

Given $Y$, $D_1$, $D_2$ and $\beta$ as in Section 
\ref{gw_side}, we explain how to construct a quiver $Q_\beta^{Y/D_1}$, that is, a 
finite oriented graph, and a dimension vector 
$\bd(\beta) \in \Z^{(Q_\beta^{Y/D_1})_0}$, where 
$(Q_\beta^{Y/D_1})_0$ denote the set of vertices of $Q_\beta^{Y/D_1}$.
In fact, our construction of $Q_\beta^{Y/D_1}$ and $\bd(\beta)$ will depend on an extra choice, 
that for simplicity we do not include in the notation. 
Different choices will define mutation equivalent quivers.

Let $Z$ be the surface obtained from 
$Y$ by blowing-up $\beta \cdot D_2$ points on $D_2$, distinct from each other and distinct from 
$D_1 \cap D_2$. Denote 
\[\pi_Y \colon Z \rightarrow Y\] 
the blow-up morphism, $F_1, \dots, F_{\beta \cdot D_2}$ the exceptional divisors and \[\beta_Z
\coloneqq \pi^{*}_Y \beta -\sum_{j=1}^{\beta \cdot D_2}
[F_j] \in H_2(Z,\Z)\,.\] 
We still denote $D_1$ and $D_2$ the strict transforms of $D_1$ and $D_2$ in $Z$.
We denote $D \coloneqq D_1 \cup D_2$.

The pair $(Z,D)$ is a log Calabi-Yau surface with maximal boundary in the sense 
of Section \ref{basics_log_cy}. 
Indeed, it is already true for $(Y,D)$ by assumption and $(Z,D)$ is obtained from $(Y,D)$ by interior blow-ups.

According to Proposition \ref{prop_toric_model}, there exists a diagram of log Calabi-Yau surfaces with maximal boundary
\begin{center}
\begin{tikzcd}
& (\tilde{Z},\tilde{D}) \arrow{dl}[swap]{\pi_Z} \arrow{dr}{\pi_{\bar{Z}}}\\
(Z,D) & & (\bar{Z},\bar{D}) \,,
\end{tikzcd}
\end{center}
where $(\bar{Z},\bar{D})$ is a toric log Calabi-Yau surface with maximal 
boundary,
$\pi_Z$ is a sequence of corner blow-ups of 
$(Z,D)$, and $\pi_{\bar{Z}}$ is a sequence of 
interior blow-ups of $(\bar{Z},\bar{D})$.
We still denote $D_1$ and $D_2$ for the 
strict transforms of $D_1$ and $D_2$ in
$\tilde{Z}$ and for the images in 
$\bar{Z}$ of these strict transforms.

From now on, we fix a choice of such diagram. Our construction of $Q_\beta^{Y/D_1}$ will depend on this choice. 
Different choices will define 
mutation equivalent quivers.

As $(\bar{Z},\bar{D})$ is toric, it admits a fan 
in $\R^2$, whose rays are in bijection with the 
irreducible components of $\bar{D}$. The morphism 
$\pi_{\bar{Z}}$ is a sequence of interior blow-ups. 
As we will be ultimately interested only in the
Gromov-Witten theory of $(Y,D)$, which is deformation invariant, we can assume
that $\pi_{\bar{Z}}$ is a blow-up of 
$\bar{Z}$ in distinct smooth points $x_1,\dots,x_n$ of 
$\bar{D}$.
Indeed, a log Calabi-Yau surface for which one has to blow-up infinitely near points are non-generic 
and can always be deformed into one for which one has to only blow-up distinct points.
Let $E_{x_1},\dots,E_{x_n}$ be the corresponding exceptional divisors in $\tilde{Z}$.

We now construct the quiver $Q_\beta^{Y/D_1}$. By definition, 
$Q_{\beta}^{Y/D_1}$ has $n$ vertices,
in natural bijection with the 
points $x_1,\dots,x_n$.
For every $1\leqslant j \leqslant n$, let $m_j \in \Z^2$ be the primitive generator, pointing away from $0$, of the ray in the fan of $\bar{Z}$ dual to the irreducible 
component of $\bar{D}$ containing $x_j$. 
Let \[\langle -,-\rangle  \rightarrow \Z^2 \times \Z^2 \rightarrow \Z\]
\[((a,b),(a',b')) \mapsto \det((a,b),(a',b'))=ab'-a'b \]
be the standard non-degenerate skew-symmetric bilinear form 
on $\Z^2$. By definition, $Q_\beta^{Y/D_1}$ has 
$\max (\langle m_j,m_k \rangle, 0)$ arrows from the vertex $j$ to the vertex $k$.

We now construct the dimension vector $\bd(\beta)$.
A dimension vector for $Q_\beta^{Y/D_1}$ is an element in $\Z^{(Q_\beta^{Y/D_1})_0}$,
where $(Q_\beta^{Y/D_1})_0$ is the set of vertices of $Q_\beta^{Y/D_1}$. 
So in order to define $\bd(\beta)$, we have to define a non-negative integer 
$\bd(\beta)_j$ for every 
$1 \leqslant j \leqslant n$. We define 
\[\bd(\beta)_j \coloneqq ((\pi_Z)^{*} \beta_Z) \cdot E_{x_j} \,.\]
We have $\beta_Z \in H_2(Z,\Z)$, $(\pi_Z)^{*} \beta_Z
\in H_2(\tilde{Z},\Z)$, and so the intersection number with the divisor
$E_{x_j}$ of $\tilde{Z}$ makes sense.
We have $\bd(\beta) \in \Z^{(Q_\beta^{Y/D_1})_0}$ but a priori 
one could 
have $\bd(\beta) \notin \NN^{(Q_\beta^{Y/D_1})_0}$.

Remark that the strict transforms in $\tilde{Z}$ of the exceptional divisors 
$F_1,\dots,F_{\beta \cdot D_2}$
naturally form a subset of the set of exceptional divisors $E_1, \dots, E_n$.
As $\beta_Z\cdot D_j=1$, it follows that there is always a subset of cardinal $\beta \cdot D_2$ of the set of 
vertices of $Q_\beta^{Y/D_1}$, over which the dimension vector $\bd(\beta)$ takes the value $1$. 
In particular, as $\beta \cdot D_2>0$, $\bd(\beta)$ is always a primitive element of 
$\Z^{(Q_\beta^{Y/D_1})_0}$.

\subsection{The quiver Donaldson-Thomas side}
We refer to 
\cite{kontsevich2008stability},
\cite{MR2951762},
\cite{MR2650811},
\cite{MR2801406},
\cite{meinhardt2017donaldson}
for Donaldson-Thomas theory of quivers.

Let $Q$ be an acyclic quiver, that is, a quiver without oriented cycles. Denote by $Q_0$ its set of vertices, and $N=\Z^{Q_0}=\oplus_{i\in Q_0} \Z e_i$.
Every $\theta=(\theta_j)_{j \in Q_0}
\in \Hom(N,\Z)$ defines a notion of stability for representations of $Q$. For every 
dimension vector $\mathbf{d} \in N$, we then have a 
projective variety $M_{\bd}^{\theta-ss}$, moduli space of $\theta$-semistable representations of $Q$ of dimension $\bd$,
containing the open smooth locus 
$M_\bd^{\theta-st}$ of $\theta$-stable representations. Denote 
$\iota \colon M_\bd^{\theta-st} \hookrightarrow M_\bd^{\theta-ss}$ the natural inclusion. 
Let $\{-,-\}$ be the skew-symmetric form on $N$ defined by $\{e_j,e_k\}:=a_{jk}-a_{kj}$, where $a_{jk}$ is the number of arrows from $j$ to $k$ in $Q$. 
Following \cite[\S 2.2]{meinhardt2017donaldson}, we say that a stability $\theta$ is $\infty$-generic\footnote{Here, $\infty$ is the slope of $-\theta(p)+i$ when $\theta(p)=0$. } if for every $\mathbf{d}_1,\mathbf{d}_2\in N$,  $\theta(\mathbf{d}_1)=\theta(\mathbf{d}_2)=0$ implies $\{\mathbf{d}_1, \mathbf{d}_2 \}=0$.

The main result of \cite{meinhardt2017donaldson} is that the Laurent polynomials
\[ \Omega_\bd^{Q_\fm, \theta}
(q^{\frac{1}{2}})
\coloneqq 
(-1)^{\dim M_\bd^{\theta-ss}} q^{-\frac{1}{2}
\dim M_\bd^{\theta-ss}}
\sum_{j=0}^{\dim M_\bd^{\theta-st}}
\left( \dim H^{2j} (M_\bd^{\theta-ss}, \iota_{!*}\Q)\right)
q^j \]
\[\in  (-1)^{\dim M_\bd^{\theta-ss}}q^{-\frac{1}{2}
\dim M_\bd^{\theta-ss}} \NN[q] \]
are the refined Donaldson-Thomas invariants of $Q$ for an $\infty$-generic stability $\theta$.
In the above formula,
$\iota_{!*}$ is the intermediate extension functor defined by $\iota$ and so
$\iota_{!*} \Q$ is a perverse sheaf on
$M_\bd^{\theta-ss}$.

We apply the previous definitions to the quiver $Q_\beta^{Y/D_1}$ constructed 
in Section \ref{quiver_construction}.
If $Q_{\beta}^{Y/D_1}$ is acyclic and $\bd(\beta) \in \NN^{(Q_{\beta}^{Y/D_1})_0}$,
we define 
\[\Omega_\beta^{Y/D_1}(q^{\frac{1}{2}})
\coloneqq \Omega_{\bd(\beta)}^{Q_{\beta}^{Y/D_1},\theta}(q^{\frac{1}{2}})\,, \]
where $\theta$ is the \emph{anti-attractor stability} given by
$\theta:= \{-,\bd(\beta)\}$. For the quiver $Q_{\beta}^{Y/D_1}$, the skew-symmetric $\{-,-\}$ is of rank $2$, and so the anti-attractor stability $\theta:= \{-,\bd(\beta)\}$ is $\infty$-generic.
If $\bd(\beta) \notin \NN^{(Q_{\beta}^{Y/D_1})_0}$, we set 
\[\Omega_\beta^{Y/D_1}(q^{\frac{1}{2}}) \coloneqq 0\,.\]

\subsection{Statement of the main result}
Using the notations introduced in the previous Sections, we can state our main result, cited as Theorem \ref{main_thm_precise_intro} in the Introduction.

\begin{thm}\label{main_thm_precise}
If the quiver $Q_{\beta}^{Y/D_1}$ is acyclic, 
then we have an equality of formal power series in $\hbar$:
\[ \Omega_\beta^{Y/D_1}(q^{\frac{1}{2}})
=(-1)^{\beta \cdot D_1+1} 
\left( 2 \sin \left( \frac{\hbar}{2}
\right) \right) 
\left( 
\frac{\hbar}{2 \sin \left( \frac{\hbar}{2} \right)}
\right)^{\beta \cdot D_2}
\left( \sum_{g\geqslant 0} N_{g,\beta}^{Y/D_1} 
\hbar^{2g-1} \right) \,,\]
where $q=e^{i \hbar}=\sum_{n \geqslant 0} \frac{(i\hbar)^n}{n!}$.
\end{thm}

Taking the leading order term on both sides of Theorem
\ref{main_thm_precise} in the limit $\hbar \rightarrow 0$,
$q^{\frac{1}{2}} \rightarrow 1$, we get the following Corollary,
cited as Theorem \ref{main_thm_0} in the Introduction.

\begin{cor}
If the quiver $Q_{\beta}^{Y/D_1}$ is acyclic, 
then we have 
\[ \Omega_\beta^{Y/D_1}
=(-1)^{\beta \cdot D_1+1} N_{0,\beta}^{Y/D_1} \,.\]
\end{cor}

\section{Exchange of absolute with relative point conditions}
\label{section_deg_point}

We start the proof of Theorem \ref{main_thm_precise}. 
Our goal is to fit the 
invariants $N_{g,\beta}^{Y/D_1}$ into the framework of 
Section 8.5 of \cite{bousseau2018quantum_tropical}.
In the present Section, we realize the first step: we exchange the $\beta \cdot D_2$ absolute point 
conditions entering the definition of 
$N_{g,d}^{Y/D_1}$ with $\beta \cdot D_2$ point conditions relative to
the divisor $D_2$.

\subsection{Statement}
We refer to \cite{MR3011419},
\cite{MR3224717}, \cite{MR3257836}, \cite{abramovich2017decomposition}
for the theory of stable log maps.
We denote $D \coloneqq D_1 \cup D_2$. We view $Y$ as a 
smooth log scheme, with the divisorial log structure defined 
by $D$.
We fix $\beta \in H_2(Y,\Z)$ such that 
$\beta \cdot D_1 >0$ and $\beta \cdot D_2 > 0$.
Let 
$\overline{M}_g(Y/D,\beta)$ be the moduli space of 
genus $g$ 
class $\beta$ 
stable log maps to $Y$, with contact order $\beta \cdot D_1$ to $D_1$ in a single point, 
and with contact order $1$ to $D_2$ in $\beta \cdot D_2$ points.
It is a proper Deligne-Mumford stack, admitting a virtual fundamental class
\[[\overline{M}_g(Y/D,\beta)]^\virt 
\in A_{g+\beta \cdot D_2}(\overline{M}_g(Y/D,\beta),\Q)\,.\]
Let $\ev_k \colon \overline{M}_g(Y/D,\beta)
\rightarrow D_2$, for 
$k=1,\dots, \beta \cdot D_2$, be the evaluation maps at the 
$\beta \cdot D_2$ contact points with $D_2$. 
We define 
\[ N_{g,\beta}^{Y/D}
\coloneqq \int_{[\overline{M}_g(Y/D,\beta)]^\virt}
(-1)^g \lambda_g \prod_{k=1}^{\beta \cdot D_2} \ev_k^{*}(\pt) \in \Q \,,\]
where $\pt \in A^1(Y)$ is the class of a point and 
$\lambda_g$ is the top Chern class of the Hodge bundle 
over $\overline{M}_g(Y/D,\beta)$.

\begin{prop} \label{prop_deg_point}
We have 
\[\sum_{g \geqslant 0}N_{g,\beta}^{Y/D_1} \hbar^{2g-1}
=\left( \sum_{g \geqslant 0} N_{g,\beta}^{Y/D} \hbar^{2g-1+\beta \cdot D_2} \right) \left( \frac{1}{\hbar}\right)^{\beta \cdot D_2} \,.\]
\end{prop}

Exchanging absolute and relative Gromov-Witten invariants by some
degeneration argument is standard in Gromov-Witten theory relative to a smooth divisor, see for example \cite{maulik2006topological}.
In particular, Proposition \ref{prop_deg_point} has exactly the form given by the usual degeneration formula 
\cite{MR1938113} applied for invariants defined relatively to $D_2$. 
The issue is that our invariants have also tangency conditions along $D_1$ and so we
are really dealing with log Gromov-Witten invariants relative to the normal crossing divisor $D=D_1 \cup D_2$. In order to prove Proposition \ref{prop_deg_point}, we have to show that the 
a priori possible corrections to the naive degeneration formula coming from 
the intersection points $D_1 \cap D_2$ actually are not there.

\subsection{Summary of the proof of Proposition \ref{prop_deg_point}}
The proof of Proposition \ref{prop_deg_point} takes the rest of Section \ref{section_deg_point}.
In Section \ref{degeneration}, we describe the relevant degeneration 
to the normal cone of $D_2$. The special fiber of the degeneration has a natural
log structure whose tropicalization is described in Section
\ref{trop_special}. This gives us a way to describe the relevant log Gromov-Witten 
invariants of the special fiber in Section \ref{inv_special}. The core of the proof is Section
\ref{tropical_curves}, which is a study of the possible tropicalizations of stable log maps contributing to these invariants. 
In Section \ref{decomp_formula}, we apply to our specific context the general decomposition formula of \cite{abramovich2017decomposition}, and in Section
\ref{rigid_tropical} we identify using the results of Section
\ref{tropical_curves} the relevant rigid tropical curves.
We conclude the proof in Section \ref{end_proof}, referring to the gluing techniques already used in
\cite{bousseau2017tropical} and \cite{bousseau2018quantum_tropical}.

\subsection{Degeneration} \label{degeneration}
Let $X$ be the degeneration of $Y$ to the normal cone of $D_2$, that is, the blow-up of 
$D_2 \times \{0\}$ in $Y \times \A^1$. Let
$\pi \colon X \rightarrow \A^1$ be the natural projection. 
We have 
$\pi^{-1}(t) \simeq Y$ if $t \neq 0$. 
Let $N_{D_2|Y}$ be the normal line bundle of $D_2$ in $Y$, and let 
$\PP$ be the projective bundle over $D_2$ obtained as the projectivization 
of $N_{D_2|Y}\oplus \cO_{D_2}$. If $p \in D_2$, we denote 
$\PP^1_p$ the fiber of the projection
$\PP \rightarrow D_2$ over $p$.
The embeddings  
$\cO_{D_2} \hookrightarrow N_{D_2|Y}\oplus \cO_{D_2}$
and $N_{D_2|Y} \hookrightarrow N_{D_2|Y}\oplus \cO_{D_2}$
define two sections of $\PP \rightarrow D$, that we denote 
$D_{2,0}$ and $D_{2,\infty}$ respectively.
The special fiber $\pi^{-1}(0)$ has two irreducible components,
$Y$ and $\PP$, with $D_2$ in $Y$ glued with $D_{2,0}$ in 
$\PP$.

We also degenerate the $\beta \cdot D_2$ point conditions. For $j=1,\dots,\beta \cdot D_2$, 
let $\sigma_j$ be a section of 
$Y \times \A^1$ such that $\sigma_j(t) \notin D_1 \cup D_2$ for $t \neq 0$ and 
$\sigma_j(0) \in D_2-(D_1 \cap D_2)$. 
We choose these sections such that $\sigma_j(0) \neq \sigma_k(0)$
if $j \neq k$, and such that $\sigma_j$ is transverse to $D_2 \times \{0\}$. For 
$j=1,\dots, \beta \cdot D_2$, we denote 
$\tilde{\sigma}_j$ the section of $\pi \colon X \rightarrow \A^1$
obtained as strict transform of $\sigma_j$. By construction, the 
$\beta \cdot D_2$ points $\tilde{\sigma}_j(0)$
are in distinct fibers of $\PP \rightarrow D_2$ and are away from 
$D_{2,0} \cup D_{2,\infty} \cup \PP^1_{p_1} \cup \PP^1_{p_2}$, where $p_1$
and $p_2$ are the two intersection points of $D_1$ with $D_2$.
For $1 \leqslant j \leqslant \beta \cdot D_2$, we denote $\PP^1_{\tilde{\sigma}_j(0)}$
the fiber of $\PP \rightarrow D_2$ passing through $\tilde{\sigma}_j(0)$.

Let $\widetilde{D_1 \times \A^1}$ be the divisor on $X$ obtained as strict transform of 
the divisor $D_1 \times \A^1$ in $Y \times \A^1$. We view $X$ as a log scheme for 
the divisorial log structure defined by the normal crossing divisor 
$\widetilde{D_1 \times \A^1} \cup \pi^{-1}(0)$. We view $\A^1$ as a log scheme for the 
toric divisorial log structure. Then $\pi$ naturally defines a log smooth morphism.
Remark that for $t \neq 0$, $\pi^{-1}(t)$ with the log structure restricted from 
$X$ is $Y$ with the divisorial log structure defined by $D_1$.

We denote $X_0 \coloneqq \pi^{-1}(0)$ with the log structure restricted from $X$.
Let $\pt_{\NN}$ be the standard log point, that is, the point $0 \in \A^1$ with the log structure 
restricted from $\A^1$.
Then $\pi$ induces by restriction a log smooth morphism 
$\pi_0 \colon X_0 \rightarrow \pt_\NN$.

\subsection{Tropicalization of the special fiber} \label{trop_special}
We refer to Appendix B of
\cite{MR3011419} 
and Section 2 of \cite{abramovich2017decomposition}
for the general notion of tropicalization of a log scheme.

Tropicalizing the log morphism $\pi_0
\colon X_0
\rightarrow \pt_\NN$, we get a morphism 
of cone complexes
$\Sigma(\pi_0)
\colon
\Sigma(X_0)
\rightarrow \Sigma(\pt_\NN)$.
We have 
$\Sigma(\pt_\NN) = \R_{\geqslant 0}$
and $\Sigma(X_0)$ is naturally
identified with the cone over the fiber 
$\Sigma(\pi_0)^{-1}(1)$ at $1 \in \R_{\geqslant 0}$. 
It is thus enough to describe the cone complex 
$\Sigma(\pi_0)^{-1}(1)$.
We denote 
\[X_0^\trop \coloneqq\Sigma(\pi_0)^{-1}(1) \,.\]
The cone complex $X_0^\trop$
consists of:
\begin{itemize}
    \item Vertices $v_{Y}$ and 
$v_{\PP}$ respectively dual to $Y$
and $\PP$.
    \item Unbounded edges 
$e_{D_1}$, $e_{{\PP^1_{p_1}}}$, $e_{{\PP^1_{p_2}}}$
respectively dual to $D_1$, $\PP^1_{p_1}$, $\PP^1_{p_2}$.
    \item One bounded edge $e_{D_2}$ dual to $D_2$.
    \item Faces
$f_{p_1}$ and $f_{p_2}$ dual to 
$p_1$ and $p_2$.
\end{itemize}

It is important to realize that $X_0^\trop$ is an abstract cone complex, 
with no natural embedding in the plane. In particular, in $X_0^\trop$, the 
unbounded edge $e_{D_1}$ is parallel to both $e_{\PP^1_{p_1}}$ and 
$e_{\PP^1_{p_2}}$.
We draw below a picture of $X_0^\trop$, where the two copies of $e_{D_2}$
and the two copies of $v_{\PP}$ have to be identified.



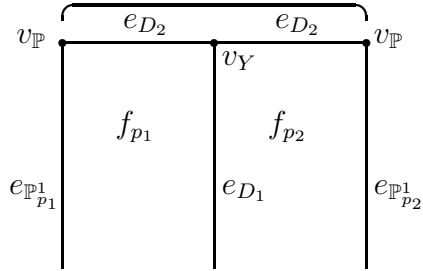
\begin{center}
\setlength{\unitlength}{1cm}
\begin{picture}(10,5)
\thicklines
\put(3,4){\circle*{0.1}}
\put(5,4){\circle*{0.1}}
\put(7,4){\circle*{0.1}}
\put(3,4){\line(1,0){2}}
\put(5,4){\line(1,0){2}}
\put(3,4){\line(0,-1){3}}
\put(5,4){\line(0,-1){3}}
\put(7,4){\line(0,-1){3}}
\put(2.3,2){$e_{{\PP^1_{p_1}}}$}
\put(7.1,2){$e_{{\PP^1_{p_2}}}$}
\put(2.4,4){$v_{\PP}$}
\put(7.1,4){$v_{\PP}$}
\put(5.1,3.7){$v_Y$}
\put(3.8,4.2){$e_{D_2}$}
\put(5.8,4.2){$e_{D_2}$}
\put(5.1,2){$e_{D_1}$}
\put(3.7,2.8){$f_{p_1}$}
\put(5.7,2.8){$f_{p_2}$}
\put(5,4.3){\oval(4,0.4)[t]}
\end{picture}
\end{center}
Remark that $\Sigma(\pi_0)^{-1}(0)$ is the ``asymptotic version" of $X_0^\trop$,
and is a ray
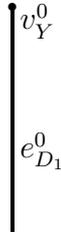
\begin{center}
\setlength{\unitlength}{1cm}
\begin{picture}(10,4)
\thicklines
\put(5,4){\circle*{0.1}}
\put(5,4){\line(0,-1){3}}
\put(5.1,3.7){$v_Y^0$}
\put(5.1,2){$e_{D_1}^0$}
\end{picture}
\end{center}
which can be identified with the tropicalization of $Y$, viewed as a
log scheme for the divisorial log structure defined by $D_1$.

Most of the tropical arguments of the following Sections will take place in 
$X_0^\trop$ and will use the above picture of $X_0^\trop$. In particular, 
the terms ``vertical", ``horizontal", ``ordinates" will refer to the corresponding notions in this
picture.

Contact orders for stable log maps to a log scheme are integral points in the 
tropicalization of this log scheme (see Definition 2.3.12 of \cite{abramovich2017decomposition}).
We denote $u_p$ the contact order for stable log maps to $X_0$ defined by $\beta \cdot D_1$ times the integral generator 
(pointing away from $v_Y^0$) of the edge 
$e_{D_1}^0$ in $\Sigma(\pi_0)^{-1}(0)$.

\subsection{Invariants of the special fiber} \label{inv_special}

Let $\overline{M}_g(X_0,\beta)$ be the moduli space of $\beta \cdot D_2$-pointed genus $g$ class $\beta$ stable 
log maps to $\pi_0 \colon X_0 \rightarrow \pt_\NN$, with a single point of contact order $u_p$.
Let 
\[[\overline{M}_g(X_0,\beta)]^{\virt} \in A_{g+2(\beta \cdot D_2)}(\overline{M}_g(X_0,\beta),\Q)\]
be its virtual fundamental class.

For simplicity, we work below in the category of stacks. In general, one should work in the category of fs log stacks to correctly define point constraints in log Gromov-Witten theory, see Section 6.3 of \cite{abramovich2017decomposition}. In our case, it does not matter because the log morphisms 
$\tilde{\sigma}_j(0) \colon \pt_{\NN} \rightarrow X_0$ are strict, that is, 
the log structure on $\pt_{\NN}$ is the pullback by $\tilde{\sigma}_j(0)$ of the log structure on $X_0$.

The $\beta \cdot D_2$ marked points define an evaluation morphism
\[\ev \colon \overline{M}_g(X_0,\beta) \rightarrow (X_0)^{\beta \cdot D_2} \,.  \]
Denote 
\[ \iota_{\tilde{\sigma}} \colon 
\tilde{\sigma}(0)=\{
(\tilde{\sigma}_1(0),\dots,\tilde{\sigma}_{\beta \cdot D_2}(0) ) \} \hookrightarrow (X_0)^{\beta \cdot D_2} \,,\]
\[\overline{M}_g(X_0,\beta,\tilde{\sigma}) \coloneqq \overline{M}_g(X_0,\beta) \times_{(X_0)^{\beta \cdot D_2}} 
\tilde{\sigma}(0) \,,\]
and 
\[[\overline{M}_g(X_0,\beta,\tilde{\sigma})]^{\virt} \coloneqq \iota_{\tilde{\sigma}}^!
[\overline{M}_g(X_0,\beta)]^{\virt} 
\in A_g(\overline{M}_g(X_0,\beta,\tilde{\sigma}),\Q)\,.\]
By deformation invariance of log Gromov-Witten invariants,
we have 
\[ N_{g,\beta}^{Y/D_1} = \int_{[\overline{M}_g(X_0,\beta,\tilde{\sigma})]^\virt}
(-1)^g \lambda_g \,.\]

\subsection{Tropical curves} \label{tropical_curves}
Let
\begin{center}
\begin{tikzcd}
C \arrow{r}{f} \arrow{d}{\nu}
& X_0 \arrow{d}{\pi_0}\\
W \arrow{r}{g} & \pt_\NN \,,
\end{tikzcd}
\end{center}
be an element of $\overline{M}_g(X_0,\beta,\tilde{\sigma})$.
Let 
\begin{center}
\begin{tikzcd}
\Sigma(C) \arrow{r}{\Sigma(f)} \arrow{d}{\Sigma(\nu)}
& \Sigma(X_0) \arrow{d}{\Sigma(\pi_0)}\\
\Sigma(W) \arrow{r}{\Sigma(g)} & \Sigma(\pt_\NN)
\end{tikzcd}
\end{center}
be its tropicalization.
For every $b \in \Sigma(g)^{-1}(1)$,
let 
\[ \Sigma(f)_b \colon
\Sigma(C)_b \rightarrow \Sigma(\pi_0)^{-1}
(1) =X_0^\trop \]
be the fiber of $\Sigma(f)$ over $b$.

The definition of $\overline{M}_g(X_0,\beta,\tilde{\sigma})$ fixes the set of unbounded edges of 
$\Sigma(C)_b$. The $\beta \cdot D_2$ marked points are dual to 
$\beta \cdot D_2$ unbounded edges $E_1,\dots, E_{\beta \cdot D_2}$ of $\Sigma(C)_b$.
The contact order $u_p$ at a single point is dual to an unbounded edge 
$E_0$ of $\Sigma(C)_b$. Furthermore, $\Sigma(C)_b$ has no other unbounded edge.

The fact that the marked points have to be mapped to $\tilde{\sigma}_1(0),
\dots,\tilde{\sigma}_{\beta \cdot D_2}(0)$ implies that the unbounded edges 
$E_j$, $1 \leqslant j \leqslant \beta \cdot D_2$, are all contracted by $\Sigma(f)_b$
onto the vertex $v_\PP$. The contact order $u_p$
implies that $E_0$ is mapped by $\Sigma(C)_b$ onto an unbounded ray in 
$X_0^\trop$, going down in a way parallel to 
$e_{\PP^1_{p_1}}$, $e_{D_1}$ and $e_{\PP^1_{p_2}}$.
The fact that the image of $E_0$ by $\Sigma(C)_b$ can a priori
be any ray parallel to $e_{D_1}$ reflects the fact that 
a stable  log map to $Y$ with a single point of maximal tangency along $D_1$
degenerates in a stable log map to $X_0$ with ``maximal tangency"
(in some logarithmic sense) along $D_1 \cup \PP^1_{p_1} \cup \PP^1_{p_2}$,
and a priori not only along $D_1$. 

The tropical balancing condition has to be satisfied everywhere except at the vertices $v_Y$ 
and $v_{\PP}$ (the corresponding components do not have toric divisorial log structures).
At $v_{\PP}$, the tropical balancing condition is only modified in the 
horizontal direction: it still holds in the vertical direction.
This follows from the general form of the
balancing conditions for stable log maps
given in Proposition 1.15 of \cite{MR3011419}.

For all $1 \leqslant j \leqslant \beta \cdot D_2$, we denote
$V_j$ the vertex of $\Sigma(C)_b$ to which the unbounded edge $E_j$
is attached. A priori, one could have $V_j = V_k$ for some 
$j \neq k$. We have $\Sigma(f)_b(V_j)=v_\PP$
for all $1 \leqslant j \leqslant \beta \cdot D_2$.
We denote $V_0$ the vertex of $\Gamma$ to which the 
unbounded edge $E_0$ is attached.

We refer to Definition 2.5.3 of \cite{abramovich2017decomposition} for details 
on edge marking.
For $V$ a vertex of $\Sigma(C)_b$ and $E$ an edge adjacent to 
$V$, we denote $v_{V,E}$ the slope measured from $V$
of the restriction of $\Sigma(f)_b$ to $E$. Given our picture of 
$X_0^\trop$, we can identify $v_{V,E}$ with a well-defined element of $\Z^2$, except if 
$\Sigma(f)_b(E)$ is contained in $e_{D_2}$, in which case $v_{V,E}$
defines two elements of $\Z^2$ corresponding to the two copies of $e_{D_2}$ and identified 
by the map gluing the two copies of $e_{D_2}$.

Using the same identification with $\Z^2$, we can write $u_p=(-\beta \cdot D_1,0)$, and we will use the standard scalar product of vectors in $\Z^2$. In particular, 
$v_{V,E} \cdot u_p=0$ is equivalent to $\Sigma(f)_b$ horizontal, 
$v_{V,E} \cdot u_p>0$ is equivalent to $\Sigma(f)_b$ pointing downward from $V$,
and $v_{V,E} \cdot u_p<0$ is equivalent to $\Sigma(f)_b$ pointing upward from $V$.

\begin{prop} \label{prop_trop_curves}
Let
\begin{center}
\begin{tikzcd}
C \arrow{r}{f} \arrow{d}{\nu}
& X_{0} \arrow{d}{\pi_0}\\
W \arrow{r}{g} & \pt_\NN \,,
\end{tikzcd}
\end{center}
be an element of $\overline{M}_g(X_0,\beta,\tilde{\sigma})$.
Let 
\begin{center}
\begin{tikzcd}
\Sigma(C) \arrow{r}{\Sigma(f)} \arrow{d}{\Sigma(\nu)}
& \Sigma(X_0) \arrow{d}{\Sigma(\pi_0)}\\
\Sigma(W) \arrow{r}{\Sigma(g)} & \Sigma(\pt_\NN) 
\end{tikzcd}
\end{center}
be its tropicalization.
Let $b$ be in the interior of $\Sigma(g)^{-1}(1)$.
Assume that $\Sigma(C)_b$ is a graph of genus $0$. Then:
\begin{itemize}
    \item For every $V$ vertex of $\Sigma(C)_b$, we have
    $\Sigma(f)_b(V)\in (e_{D_1} \cup e_{D_2})$.
    \item Let $V$ and $V'$ be vertices of $\Sigma(C)_b$ such that 
    $\Sigma(f)_b(V)$ is contained in the interior of $e_{D_1}$
    and $\Sigma(f)_b(V')$ is contained in the interior of $e_{D_2}$.
    Then, there is no edge of $\Sigma(C)_b$ connecting $V$ and $V'$.
    \item We have $E_0 \subset e_{D_1}$.
\end{itemize}
\end{prop}

The proof of Proposition 
\ref{prop_trop_curves} is given by the following sequence of Lemmas and
takes the rest of 
Section \ref{tropical_curves}.

\begin{lem} \label{lem_descent}
Let $V$ be a vertex of $\Sigma(C)_b$ and let $E$ be an edge of $\Sigma(C)_b$ adjacent to $V$ such that 
$v_{V,E} \cdot u_p>0$. Assume that $E \neq E_0$.
Then there exists a path
$\mathfrak{c}_{V,E}$
in $\Sigma(C)_b$, starting at $V$, of first edge $E$, ending at $V_0$, and whose vertices have images by 
$\Sigma(f)_b$ of strictly decreasing ordinates.
\end{lem}

\begin{proof}
The vertical projection of the balancing condition holds everywhere in $X_0^\trop$
except at $v_Y$.
It follows that, starting with $V$ and $E$ such that $v_{V,E} \cdot u_p >0$,
there exists a path in $\Sigma(C)_b$, starting at $V$, of first edge $E$ and whose vertices have images by $\Sigma(f)_b$ of strictly decreasing ordinates.
Indeed, all the vertices of the path except maybe the first one
have ordinates strictly less that the ordinate of $v_Y$ and so in particular are
distinct of $v_Y$. This path can only end at $V_0$.
\end{proof}

\begin{lem} \label{lem_V_0_lower}
There is no vertex $V$ of $\Sigma(C)_b$
such that $\Sigma(f)_b(V)$ has an ordinate strictly less than the ordinate of 
$\Sigma(f)_b(V_0)$.
\end{lem}

\begin{proof}
Let $V$ be such vertex. We have necessarily $\Sigma(f)_b(V) \neq v_Y$ and so
the vertical projection of the balancing condition applies to $V$. 
It follows from Lemma \ref{lem_descent} that every edge $E$ adjacent to 
$V$ has $v_{V,E}=0$. In particular, every vertex adjacent to $V$ has the same ordinate than $V$ and the previous argument applies to it. It follows that the vertices of 
$\Sigma(C)_b$ whose images by $\Sigma(f)_b$ have ordinates equal to the 
ordinate of $\Sigma(f)_b(V)$ define a subgraph of 
$\Sigma(C)_b$, union of connected components. This contradicts the fact that $\Sigma(C)_b$ is connected.
\end{proof}

\begin{lem} \label{lem_D_2}
Let $V$ be a vertex of $\Sigma(C)_b$ such that 
$\Sigma(f)_b(V) \in e_{D_2}$ and 
$\Sigma(f)_b \neq v_Y$. Let $E$ be an edge 
of $\Sigma(C)_b$ adjacent to $V$. Then 
$E$ is horizontal, that is, 
$v_{V,E}\cdot u_p=0$.
\end{lem}

\begin{proof}
If $v_{V,E}\cdot u_p >0$ (resp.\ $<0$), then, by the vertical balancing condition 
(applied through the gluing map of our two copies of $e_{D_2}$), 
there would exist an edge $E'$, adjacent to $V$ such that $v_{V,E'}\cdot u_p <0$
(resp. $>0$). Applying Lemma
\ref{lem_descent} to both $E$ and $E'$, one would get two distinct paths $\mathfrak{c}_{V,E}$ 
and $\mathfrak{c}_{V',E'}$ starting at $V$ and ending at $V_0$. The union 
$\mathfrak{c}_{V,E} \cup \mathfrak{c}_{V,E'}$ would be a non-trivial closed cycle in $\Sigma(C)_b$, 
in contradiction with our assumption that 
$\Sigma(C)_b$ has genus $0$.
\end{proof}

Recall that $\PP$ has a natural fibration structure $\PP \rightarrow D_2$.

\begin{lem} \label{lem_no_h_component}
Let $C'$ be an irreducible component of $C$ such that 
$f(C') \subset \PP$. Then $f(C')$ is contained in a fiber of $\PP \rightarrow D_2$.
\end{lem}

\begin{proof}
Let $V$ be the vertex of $\Sigma(C)_b$
dual to $C'$. We have $\Sigma(f)_b(V) \in e_{D_2}$ and $\Sigma(f)_b(V)\neq v_Y$.
Assume by contradiction that
$f(C')$ is not contained in a fiber of $\PP \rightarrow D_2$. Then
$f(C')$ intersects
$\PP^1_{p_1}$ and $\PP^1_{p_2}$. So there are edges $E_{p_1}$ and $E_{p_2}$
adjacent to $V$ such that $\Sigma(f)_b(E_{p_1})$ goes inside $f_{p_1}$ with $v_{V,E_{p_1}}\cdot u_p>0$ and
similarly $\Sigma(f)_b(E_{p_2})$ goes inside $f_{p_2}$ with $v_{V,E_{p_2}}\cdot u_p>0$.
This contradicts Lemma \ref{lem_D_2}.
\end{proof}

Let $\Gamma_\PP$ be the subgraph of 
$\Sigma(C)_b$ consisting of vertices mapped by 
$\Sigma(f)_b$ on $e_{\PP^1_{p_1}} \cup e_{\PP^1_{p_2}}$.
Let $\tau$ be a connected component of 
$\Gamma_\PP$ and let $C_{\tau}$ be the curve dual to $\Gamma_\tau$. By Lemma
\ref{lem_no_h_component}, 
$f(C_\tau)$ is contained in a fiber of 
$\PP \rightarrow D_2$. We can write 
$[f(C_\tau)]=d_\tau [\PP^1]
\in H_2(\PP,\Z)$, for some $d_\tau \in \NN$, where $[\PP^1]$ is the class of a fiber of $\PP \rightarrow D_2$. 
As $C$ is connected, $C_\tau$ cannot be entirely contracted by $f$, and so $d_\tau \geqslant 1$. 
In fact, $d_{\tau}$ is the sum of weights of edges of $\Sigma(C)_b$
adjacent to $\tau$ and mapped non-trivially in $e_{D_2}$.

For all $1 \leqslant j \leqslant \beta \cdot D_2$, let $\tau_j$ be the connected component of $\Gamma_\PP$ containing 
$V_j$. As the points 
$\tilde{\sigma}_1(0),
\dots, \tilde{\sigma}_{\beta \cdot D_2}(0)$ 
are in different fibers of 
$\PP \rightarrow D_2$, we have 
$\tau_j \neq \tau_k$ if $j \neq k$.
In particular, we have $V_j \neq V_k$
if $j \neq k$.

Let $\cV_Y$ be the set of vertices $V$ of
$\Sigma(C)_b$ such that 
$\Sigma(f)_b(V)=v_Y$, and which are connected to
some vertex in some connected component 
of $\Gamma_\PP$ by some path mapped
in $e_{D_2}$ by $\Sigma(f)_b$.
Let $C_Y$ be the union of irreducible components of $C$ dual to vertices 
in $\cV_Y$. Denote 
$\beta_Y \coloneqq [f(C_Y)]
\in H_2(Y,\Z)$.

\begin{lem} \label{lem_homology}
We have:
\begin{itemize}
    \item $\beta_Y \cdot D_2=\beta \cdot D_2$.
    \item The connected components of $\Gamma_\PP$ are exactly the $\tau_j$, 
$1 \leqslant j \leqslant \beta \cdot D_2$. In particular, no vertex of $\Sigma(C)_b$ is mapped by $\Sigma(f)_b$
in the interior of $e_{\PP^1_{p_1}}$ or the 
interior of $e_{\PP^1_{p_2}}$.
    \item $d_{\tau_j}=1$ for every 
$1 \leqslant j \leqslant \beta \cdot D_2$.
\end{itemize}
\end{lem}

\begin{proof}
It follows from Lemma \ref{lem_D_2}
and from the horizontal balancing condition, satisfied on 
$e_{D_2}$ away from $v_Y$ and $v_\PP$, that 
\[\beta_Y \cdot D_2 \geqslant \sum_{\tau} d_\tau \,,\]
where the sum is over the connected components of $\Gamma_\PP$.

By Lemma \ref{lem_no_h_component}, no irreducible component of $C$ is mapped onto $D_2$,
so the image of every irreducible
irreducible component of $C$ intersects 
$D_2$ non-negatively and so
$\beta \cdot D_2 \geqslant \beta_Y\cdot D_2$. On the other hand, we have 
$\sum_{\tau} d_\tau
\geqslant \sum_{j=1}^{\beta \cdot D_2}
d_{\tau_j} \geqslant \beta \cdot D_2$.

The combination of the previous inequalities gives 
\[\beta \cdot D_2 \geqslant \beta_Y\cdot D_2
\geqslant
\sum_{\tau} d_\tau
\geqslant \sum_{j=1}^{\beta \cdot D_2}
d_{\tau_j} \geqslant \beta \cdot D_2\,,\]
and so all these inequalities are in fact equalities. Lemma 
\ref{lem_homology} follows.
\end{proof}

\begin{lem} \label{lem_no_D_1}
Let $C'$ be an irreducible component of $C$ such that $f(C') \subset Y$ and $C'$ is not contracted by $f$. Then 
\[ (f(C') \cap D_2) \subset \{ \PP^1_{\tilde{\sigma}_j(0)} \cap D_2, 
1 \leqslant j \leqslant \beta \cdot D_2\}\,.\]
\end{lem}

\begin{proof} 
By Lemma \ref{lem_no_h_component}, no irreducible component of $C$ is mapped onto $D_2$,
so the image of every irreducible
irreducible component of $C$ intersects 
$D_2$ non-negatively.
But according to Lemma
\ref{lem_homology}, the total intersection number $\beta \cdot D_2$ of images by $f$ components of $C$ mapping to $Y$ with $D_2$
is already accounted by 
the intersection points $\PP_{\tilde{\sigma}_j(0)}^1 \cap D_2, 
1 \leqslant j \leqslant \beta \cdot D_2$.
\end{proof}

\begin{lem} \label{lem_no_f}
There is no vertex $V$ of $\Sigma(C)_b$
such that $\Sigma(f)_b(V)$ is contained in the interior of $f_{p_1}$ or the interior of $f_{p_2}$.
\end{lem}

\begin{proof}
It follows from Lemma \ref{lem_homology} and Lemma
\ref{lem_no_D_1} that no non-contracted component of $C$
has an image by $f$ intersecting $p_1$ or $p_2$.
As $C$ is connected, this implies that no component of $C$
can be contracted by $f$ on $p_1$ or $p_2$.
\end{proof}

The first two points of Proposition 
\ref{prop_trop_curves} follow from 
Lemma \ref{lem_homology}, Lemma \ref{lem_no_D_1}  and
Lemma \ref{lem_no_f}.
In particular, we have 
$\Sigma(f)_b(V_0) \in e_{D_1} \cup e_{D_2}$. As $\beta \cdot D_1 > 0$, there is at least one non-contracted component of
$C$ mapping to $Y$ and intersecting 
$D_1$, and so there is at least one vertex $V$ and an edge $E$ such that 
$v_{V,E}\cdot u_p>0$. Using Lemma 
\ref{lem_V_0_lower}, it follows that 
$\Sigma(f)_b(V_0) \in e_{D_2}$ is 
impossible, and so $\Sigma(f)_b(V_0) \in e_{D_1}$, proving the third point of Proposition \ref{prop_trop_curves}.

\subsection{Decomposition formula} \label{decomp_formula}

We refer to Definitions 4.2.1 and 4.3.1 of \cite{abramovich2017decomposition} 
for the notions of decorated parametrized tropical curve and rigid decorated 
parametrized tropical curves. We say that a decorated parametrized tropical curve 
$h \colon \Gamma \rightarrow X_0^\trop$ is of type $u_p$ if 
\begin{itemize}
    \item The sum over vertices of curve classes decorations is equal to $\beta$.
    \item $\Gamma$ contains unbounded edges $E_1,\dots, E_{\beta \cdot D_2}$ contracted 
    by $h$ onto $v_{\PP}$.
    \item $\Gamma$ contains an unbounded edge $E_0$ mapped by $h$ with weight 
    $\beta \cdot D_1$ onto an unbounded ray in $X_0^\trop$, going down and parallel to 
    $e_{D_1}$.
    \item $E_0$ and $E_1,\dots, E_{\beta \cdot D_2}$ are the only unbounded edges of $\Gamma$.
\end{itemize}

Let $h \colon \Gamma \rightarrow X_0^\trop$ be a rigid genus $g$ 
decorated parametrized tropical curve of type $u_p$.
Let $\overline{M}_g^h(X_0,\beta)$ be the moduli space of genus $g$ stable log maps of type $u_p$ marked by $h$, that is, equipped of a retraction of their tropicalization onto $h$.
According to Proposition 4.4.2 of \cite{abramovich2017decomposition}, it is a proper Deligne-Mumford stack, equipped with a virtual fundamental class 
$[\overline{M}_g^h(X_0,\beta)]^{\virt}$.
The $\beta \cdot D_2$ marked points define an evaluation morphism
\[\ev \colon \overline{M}_g^h(X_0,\beta) \rightarrow (X_0)^{\beta \cdot D_2} \,.  \]
Denote 
\[ \iota_{\tilde{\sigma}}^h \colon 
\tilde{\sigma}(0)=\{
(\tilde{\sigma}_1(0),\dots,\tilde{\sigma}_{\beta \cdot D_2}(0) ) \} \hookrightarrow (X_0)^{\beta \cdot D_2} \,,\]
\[\overline{M}_g^h(X_0,\beta,\tilde{\sigma}) \coloneqq \overline{M}_g^h(X_0,\beta) \times_{(X_0)^{\beta \cdot D_2}} 
\tilde{\sigma}(0) \,,\]
and 
\[[\overline{M}_g^h(X_0,\beta,\tilde{\sigma})]^{\virt} \coloneqq \iota_{\tilde{\sigma}}^!
[\overline{M}_g^h(X_0,\beta)]^{\virt} 
\in A_g(\overline{M}_g^h(X_0,\beta,\tilde{\sigma}),\Q)\,.\]

Forgetting the marking by $h$ gives a morphism 
\[j_h \colon \overline{M}_g^h(X_0,\beta,\tilde{\sigma}) \rightarrow \overline{M}_g(X_0,\beta,\tilde{\sigma})\,.\]
According to the decomposition formula, Theorem 4.8.1 of 
\cite{abramovich2017decomposition}, we have 
\[ [\overline{M}_g(X_0,\beta,\tilde{\sigma})]^\virt
= \sum_{h}  
\frac{n_h}{|\Aut (h)|}
(j_{h })_* [\overline{M}_g^{h}(X_0,\beta,\tilde{\sigma})]^\virt \,,\]
where the sum is over 
rigid genus $g$ decorated parametrized 
tropical curves 
$h \colon \Gamma \rightarrow 
X_0^\trop$
of type $u_p$, $n_h$ is the smallest positive integer such that after scaling by $n_h$, 
$h$ gets integral vertices and integral lengths, and $|\Aut(h)|$ is the order of the automorphism group of $h$. 

\subsection{Contributing rigid tropical curves} \label{rigid_tropical}

We first explain how to construct a particular class of decorated parametrized tropical curves.
Let $\vec{g}=(g_0,g_1,\dots,g_{\beta \cdot D_2})$ be a $(\beta \cdot D_2+1)$-tuple of non-negative integers such that $|\vec{g}| \coloneqq g_0 + \sum_{j=1}^{\beta \cdot D_2}
g_j = g$.
Let $\Gamma_{\vec{g}}$ be the genus $0$ graph consisting of vertices $V_0$ and $V_j$, $1 \leqslant j \leqslant 
\beta \cdot D_2$, bounded edges 
$E_j^{D_2}$,
$1\leqslant j \leqslant \beta \cdot D_2$,
connecting $V_0$ and $V_j$, 
 unbounded edges $E_j$, $1\leqslant j \leqslant \beta \cdot D_2$, attached to $V_j$
and one unbounded edge $E_0$ attached to $V_0$.
We define a structure of tropical curve on 
$\Gamma_{\vec{g}}$ by assigning:
\begin{itemize}
    \item Genera to the vertices. We assign 
    $g_0$ to $V_0$ and $g_j$ to $V_j$.
    \item The length $\ell(E_j^{D_2})=1$ to the bounded 
    edge $E_j^{D_2}$, for all $1 \leqslant j \leqslant 
    \beta \cdot D_2$.
\end{itemize}
Finally, we define a decorated parametrized 
tropical curve 
\[h_{\vec{g}} \colon \Gamma_{\vec{g}} \rightarrow X_0^\trop\]
by the following data:
\begin{itemize}
    \item We define $h_{\vec{g}}(V_0) \coloneqq v_Y$ and 
    $h_{\vec{g}}(V_j)=v_{\PP}$ for all $1 \leqslant j \leqslant \beta \cdot D_2$.
    \item For all $1 \leqslant j \leqslant \beta \cdot D_2$, the bounded edge $E_j^{D_2}$ is mapped by $h_{\vec{g}}$ to the bounded edge $e_{D_2}$ with weight $1$.
    \item For all $1 \leqslant j \leqslant \beta \cdot D_2$, the unbounded edge $E_j$ is contracted by $h_{\vec{g}}$ to the vertex $v_{\PP}$.
    \item The unbounded edge $E_0$ is mapped by $h$ to the unbounded edge $e_{D_1}$ with weight $\beta \cdot D_1$.
    \item We decorate $V_0$ with the curve class $\beta \in H_2(Y,\Z)$.
    \item For all $1 \leqslant j \leqslant \beta \cdot D_2$, we decorate the vertex $V_j$ with the curve class 
    $[\PP^1]$, fiber class of $\PP \rightarrow D_2$.
\end{itemize}

Figure: the tropical curve $\Gamma_{\vec{g}}$.
\begin{center}
\setlength{\unitlength}{1cm}
\begin{picture}(10,4)
\thicklines
\put(3,2){\circle*{0.1}}
\put(5,2){\circle*{0.1}}
\put(3,2){\line(-1,0){2}}
\put(3,2){\line(1,0){2}}
\put(5,2){\line(1,0){2}}
\put(5,1){\circle*{0.1}}
\put(5,1){\line(1,0){2}}
\put(5,3){\circle*{0.1}}
\put(5,3){\line(1,0){2}}
\put(3,2){\line(2,1){2}}
\put(3,2){\line(2,-1){2}}
\put(1,2.2){$E_0$}
\put(2.8,2.2){$V_0$}
\put(5,3.2){$V_1$}
\put(5,2.2){$V_j$}
\put(5,1.2){$V_{\beta \cdot D_2}$}
\put(6.8,1.2){$E_{\beta \cdot D_2}$}
\put(6.8,3.2){$E_1$}
\put(6.8,2.2){$E_j$}
\put(3.8,1){$E_{\beta \cdot D_2}^{D_2}$}
\put(3.8,2.1){$E_j^{D_2}$}
\put(3.8,2.8){$E_1^{D_2}$}
\end{picture}
\end{center}

\begin{lem}
For every $\vec{g}$, the genus $g$ decorated parametrized tropical curve $h_{\vec{g}} \colon \Gamma_{\vec{g}} \rightarrow X_0^\trop$ is rigid and of type $u_p$.
\end{lem}

\begin{proof}
The fact that $h_{\vec{g}}$ is of type $u_p$ is immediate from its definition.
The rigidity follows from the facts that $h_{\vec{g}}$ has no contracted bounded edge and all vertices of $\Gamma_{\vec{g}}$ are mapped to vertices of $X_0^\trop$: it is not possible to deform $h$ without changing its combinatorial type.
\end{proof}

\begin{prop} \label{prop_rigid}
Let 
$h \colon \Gamma \rightarrow X_0^\trop$ 
be a genus $g$ rigid decorated parametrized tropical curve of type $u_p$. 
Assume that there exists 
\begin{center}
\begin{tikzcd}
C \arrow{r}{f} \arrow{d}{\nu}
& X_{0} \arrow{d}{\pi_0}\\
W \arrow{r}{g} & \pt_\NN \,,
\end{tikzcd}
\end{center}
element of $\overline{M}_g^h(X_0,\beta,\tilde{\sigma})$, such that the dual graph of 
$C$ has genus $0$. Then there exists $\vec{g}$ with 
$|\vec{g}|=g$ such that $h=h_{\vec{g}}$.
\end{prop}

\begin{proof}
Consider the tropicalization
\begin{center}
\begin{tikzcd}
\Sigma(C) \arrow{r}{\Sigma(f)} \arrow{d}{\Sigma(\nu)}
& \Sigma(X_0) \arrow{d}{\Sigma(\pi_0)}\\
\Sigma(W) \arrow{r}{\Sigma(g)} & \Sigma(\pt_\NN)  \,.
\end{tikzcd}
\end{center}
For every $b \in \Sigma(g)^{-1}(1)$,
let 
\[ \Sigma(f)_b \colon
\Sigma(C)_b \rightarrow \Sigma(\pi_0)^{-1}
(1) =X_0^\trop \]
be the fiber of $\Sigma(f)$ over $b$.
It follows from the definition of marking by $h$ (Definition 4.4.1 of 
\cite{abramovich2017decomposition}) that $h$ is a retraction of 
$\Sigma(f)_b \colon \Sigma(C)_b \rightarrow X_0^\trop$.
By assumption $\Sigma(C)_b$ has genus $0$ and so we can apply Proposition 
\ref{prop_trop_curves} to get a relatively explicit description of $\Sigma(f)_b$.
It follows from this description that any rigid retraction of $\Sigma(f)_b$
is of the form $h_{\vec{g}}$ for some $\vec{g}$.
\end{proof}

\begin{lem} \label{lem_vanishing}
Let $h \colon \Gamma \rightarrow X_0^\trop$ be a rigid genus $g$ decorated parametrized 
tropical curve of type $u_p$. Assume that for every
\begin{center}
\begin{tikzcd}
C \arrow{r}{f} \arrow{d}{\nu}
& X_{0} \arrow{d}{\pi_0}\\
W \arrow{r}{g} & \pt_\NN \,,
\end{tikzcd}
\end{center}
element of $\overline{M}_g^h(X_0,\beta,\tilde{\sigma})$, the dual graph of $C$
has positive genus.
Then we have 
\[\int_{[\overline{M}_g^{h_{\vec{g}}}(X_0,\beta,\tilde{\sigma})]^\virt} (-1)^g \lambda_g =0\,.
\]
\end{lem}

\begin{proof}
It is a general property of $\lambda_g$ that it vanishes on families of curves containing cycles of irreducible components
(see for example Lemma 8 of 
\cite{bousseau2017tropical}).
\end{proof}

\subsection{End of the proof of Proposition \ref{prop_deg_point}} \label{end_proof}
Combining Proposition \ref{prop_rigid} and Lemma \ref{lem_vanishing}, the 
decomposition formula of Section \ref{decomp_formula} implies that
\[N_{g,\beta}^{Y/D_1}
=\sum_{\vec{g},|\vec{g}|=g} \int_{[\overline{M}_g^h(X_0,\beta,\tilde{\sigma})]^{\virt}}
(-1)^g \lambda_g\,.\]
The last step of the proof of Proposition \ref{prop_deg_point} is the 
expression of 
\[\int_{[\overline{M}_g^h(X_0,\beta,\tilde{\sigma})]^{\virt}}
(-1)^g \lambda_g\]
as product of factors indexed by the vertices of $V$, and the evaluation of each factor.
This is obtained by the same gluing argument used in Section 7 of 
\cite{bousseau2017tropical} and Section 5.5 of \cite{bousseau2018quantum_tropical}.
The factors $\frac{1}{\hbar}$
are the contributions of the vertices $V_j$
of $h_{\vec{g}}$,
$1 \leqslant j \leqslant \beta \cdot D_2$,
and come from the relative Gromov-Witten theory of
$\PP^1_{\tilde{\sigma}_j(0)} \simeq \PP^1$.
Indeed, let $M$ be the moduli space of degree $1$ stable maps to 
$\PP^1$ relative to a point $\infty \in \PP^1$, $\nu \colon C \rightarrow M$ the universal curve, $f \colon C \rightarrow \PP^1$ the universal stable map, and  $[M]^{\virt}$ the natural virtual fundamental class. As the normal bundle  of $\PP^1_{\tilde{\sigma}_j(0)}$ in $\PP$ is trivial,
the virtual fundamental class on $M$ coming from the surface $\PP$ is obtained by intersecting $[M]^{\virt}$ with 
$e(R^1 \nu_{*}f^{*}\cO)=(-1)^g \lambda_g$. Therefore, 
the contribution of $V_j$ is $\sum_{g \geq 1} \left( \int_{[M]^{\virt}} \lambda_g^2 \right) \hbar^{2g-1}$. By Mumford's formula \cite{mumford}, we have 
$\lambda_g^2=0$ for $g>0$ and $\lambda_0^2=1$.

\section{Exchange of relative point conditions with blow-ups}
\label{section_exchange_blow_ups}

\subsection{Statement}
We continue the proof of Theorem \ref{main_thm_precise}. Given Proposition \ref{prop_deg_point}, it is enough to 
fit the invariants $N_{g,\beta}^{Y/D}$ into the framework of Section 5.8 of 
\cite{bousseau2018quantum_tropical}. This will be done by exchanging the $\beta \cdot D_2$ relative points conditions
along $D_2$ with no condition on a surface obtained by blowing-up $\beta \cdot D_2$ points on $D_2$.

Let $Z$ be the surface obtained from 
$Y$ by blowing-up $\beta \cdot D_2$ points on
$D_2$, which are 
distinct from each other and distinct from $D_1 \cap D_2$.
Denote $\pi_Y \colon Z \rightarrow Y$ the blow-up morphism, $F_1,\dots, F_{\beta \cdot D_2}$ the exceptional divisors and
\[ \beta_Z
\coloneqq \pi_Y^{*}\beta -\sum_{j=1}^{\beta \cdot D_2} [F_j] 
\in H_2(Z,\Z)\,.\]  We still denote $D_1$, $D_2$ and $D$ the strict transforms of $D_1$, $D_2$ and $D$ in $Z$.
We view $Z$ as a smooth log scheme for the divisorial log structure defined by $D
=D_1 \cup D_2$.

Let 
$\overline{M}_g(Z/D,\beta_Z)$ be the moduli space of genus $g$ class $\beta_Z$ stable log maps to 
$Z$, with contact order $\beta \cdot D_1$ to $D_1$ in a single point. 
It is a proper Deligne-Mumford stack, admitting a virtual fundamental class 
\[ [\overline{M}_g(Z/D,\beta_Z)]^\virt
\in A_g(\overline{M}_g(Z/D,\beta_Z),\Q)\,.\]
We define 
\[ N_{g,\beta}^{Z/D} \coloneqq \int_{
[\overline{M}_g(Z/D,\beta_Z)]^\virt} (-1)^g \lambda_g \in \Q \,.\]

\begin{prop}\label{prop_exchange_blow_ups}
We have 
\[\sum_{g \geqslant 0}N_{g,\beta}^{Z/D} \hbar^{2g-1}
=\left( \sum_{g \geqslant 0} N_{g,\beta}^{Y/D} \hbar^{2g-1+\beta \cdot D_2} \right) \left( \frac{1}{2 \sin \left(
\frac{\hbar}{2}
\right)}\right)^{\beta \cdot D_2} \,.\]
\end{prop}

\subsection{Summary of the proof of Proposition \ref{prop_exchange_blow_ups}}
The proof of Proposition \ref{prop_exchange_blow_ups} takes the rest of Section
\ref{section_exchange_blow_ups} and is parallel to the proof of 
Proposition \ref{prop_deg_point} given in Section \ref{section_deg_point}.
For this reason, we will often refer to the proof of Proposition \ref{prop_deg_point}
and only explain in detail the points specific to Proposition \ref{prop_exchange_blow_ups}. 

In Section \ref{degeneration2}, we describe the relevant degeneration. The special fiber of the degeneration has a natural
log structure whose tropicalization is described in Section
\ref{trop_special2}. This gives us a way to describe the relevant log Gromov-Witten 
invariants of the special fiber in Section \ref{inv_special2}. The core of the proof is Section
\ref{tropical_curves2}, which is a study of the possible tropicalizations of stable log maps contributing to these invariants. 
We conclude the proof in Section \ref{end_proof2}.

\subsection{Degeneration} \label{degeneration2}
As in Section \ref{degeneration}, let 
$X$ be the degeneration of $Y$ to the normal cone of $D_2$. 
For every $j=1,\dots,\beta \cdot D_2$, we choose a point $x_j$ on $D_2-(D_2\cap D_1)$ such that 
$x_j \neq x_k$ if $j \neq k$. Let $s_j$ be the section of 
$\pi \colon X \rightarrow \A^1$ obtained as strict transform of the section 
$\{x_j\} \times \A^1$ of $Y \times \A^1$.

We blow-up the sections $s_j$ of $\pi$ to obtain a new family
$\tilde{\pi} \colon \tilde{X} \rightarrow  \A^1$.
We have $\tilde{\pi}^{-1}(t)=Z$ if $t\neq 0$. 
Let $\tilde{\PP}$ be the surface obtained from $\PP=\PP(N_{D_2|Y} \oplus \cO_{D_2})$
by blowing-up the points 
$s_j(0)$, $1 \leqslant j \leqslant \beta \cdot D_2$. We still denote $D_{2,0}$, 
$D_{2,\infty}$, $\PP^1_{p_1}$ and  $\PP^2_{p_2}$
the strict transforms in $\tilde{\PP}$ of the divisors $D_{2,0}$, 
$D_{2,\infty}$, $\PP^1_{p_1}$ and  $\PP^2_{p_2}$ of $\PP$.
The special fiber $\tilde{\pi}^{-1}(0)$ has two irreducible components, 
$Y$ and $\tilde{\PP}$, with $D_2$ in $Y$
glued along $D_{2,0}$ in $\tilde{\PP}$.

Let $\widetilde{D_1 \times \A^1}$ be the divisor on $\tilde{X}$
obtained as strict transform of the divisor $D_1 \times \A^1$ in 
$Y \times \A^1$.
Similarly, 
let $\widetilde{D_2 \times \A^1}$ be the divisor on $\tilde{X}$
obtained as strict transform of the divisor $D_2 \times \A^1$ in 
$Y \times \A^1$.
 We view $\tilde{X}$ as a log scheme for the 
divisorial log structure defined by the normal crossing divisor 
$\widetilde{D_1 \times \A^1} \cup
\widetilde{D_2 \times \A^1} \cup \tilde{\pi}^{-1}(0)$.
We view $\A^1$ as a log scheme for the toric divisorial log structure.
Then $\tilde{\pi}$ naturally define a log smooth morphism. Remark that for $t \neq 0$, 
$\tilde{\pi}^{-1}(t)$ with the log structure restricted from $\tilde{X}$
is $Z$ with the divisorial log structure defined by 
$D=D_1 \cup D_2$.

We denote $\tilde{X}_0 \coloneqq \tilde{\pi}^{-1}(0)$ with the log structure restricted from $\tilde{X}$,
and $\tilde{\pi}_0 \colon \tilde{X}_0 \rightarrow \pt_\NN$ the log smooth morphism to the standard log point induced by $\tilde{\pi}$.

\subsection{Tropicalization of the special fiber} \label{trop_special2}

Tropicalizing the log morphism $\tilde{\pi}_0
\colon \tilde{X}_0
\rightarrow \pt_\NN$, we get a morphism 
of cone complexes
$\Sigma(\tilde{\pi}_0)
\colon
\Sigma(\tilde{X}_0)
\rightarrow \Sigma(\pt_\NN)$.
We have 
$\Sigma(\pt_\NN) = \R_{\geqslant 0}$
and $\Sigma(\tilde{X}_0)$ is naturally
identified with the cone over the fiber 
$\Sigma(\pi_0)^{-1}(1)$ at $1 \in \R_{\geqslant 0}$. 
It is thus enough to describe the cone complex 
$\Sigma(\tilde{\pi}_0)^{-1}(1)$.
We denote 
\[\tilde{X}_0^\trop \coloneqq\Sigma(\tilde{\pi}_0)^{-1}(1) \,.\]
The cone complex $\tilde{X}_0^\trop$ consists of:
\begin{itemize}
    \item Vertices 
$v_Y$ and $v_{\tilde{\PP}}$
respectively dual to 
$Y$ and $\tilde{\PP}$.
   \item unbounded edges 
$e_{D_1}$, $e_{\PP^1_{p_1}}$,
$e_{\PP^1_{p_2}}$, $e_{D_{2,\infty}}$
respectively dual to 
$D_1$, 
$\PP^1_{p_1}$,
$\PP^1_{p_2}$, 
$D_{2,\infty}$.
   \item One bounded edge $e_{D_{2,0}}$ dual to 
$D_{2,0}$.
   \item Faces $f_{p_1,0}$,
$f_{p_2,0}$, $f_{p_1,\infty}$,
$f_{p_2,\infty}$ respectively dual to $p_{1,0}$,
$p_{2,0}$, $p_{1,\infty}$,
$p_{2,\infty}$.
\end{itemize}

We draw below a picture of 
$\tilde{X}_0^\trop$, where the two copies 
of $e_{D_{2,0}}$, the two copies of $v_{\tilde{\PP}}$ and the two copies 
of $e_{D_{2,\infty}}$ have to be identified.

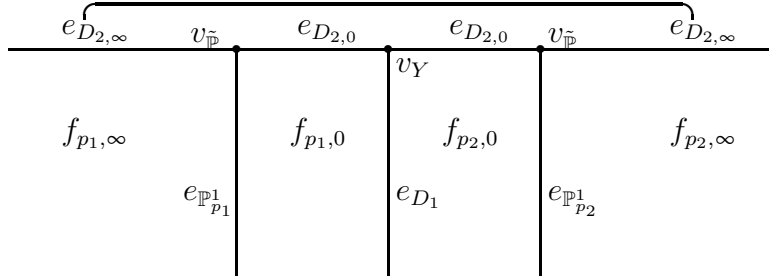
\begin{center}
\setlength{\unitlength}{1cm}
\begin{picture}(10,5)
\thicklines
\put(3,4){\circle*{0.1}}
\put(5,4){\circle*{0.1}}
\put(7,4){\circle*{0.1}}
\put(3,4){\line(1,0){2}}
\put(5,4){\line(1,0){2}}
\put(3,4){\line(0,-1){3}}
\put(5,4){\line(0,-1){3}}
\put(7,4){\line(0,-1){3}}
\put(3,4){\line(-1,0){3}}
\put(7,4){\line(1,0){3}}
\put(2.3,2){$e_{{\PP^1_{p_1}}}$}
\put(7.1,2){$e_{{\PP^1_{p_2}}}$}
\put(2.4,4.1){$v_{\tilde{\PP}}$}
\put(7.1,4.1){$v_{\tilde{\PP}}$}
\put(5.1,3.7){$v_Y$}
\put(3.8,4.2){$e_{D_{2,0}}$}
\put(5.8,4.2){$e_{D_{2,0}}$}
\put(5.1,2){$e_{D_1}$}
\put(3.7,2.8){$f_{p_1,0}$}
\put(5.7,2.8){$f_{p_2,0}$}
\put(0.7,2.8){$f_{p_1,\infty}$}
\put(8.7,2.8){$f_{p_2,\infty}$}
\put(0.7,4.2){$e_{D_{2,\infty}}$}
\put(8.7,4.2){$e_{D_{2,\infty}}$}
\put(5,4.4){\oval(8,0.4)[t]}
\end{picture}
\end{center}
Remark that $\Sigma(\tilde{\pi}_0)^{-1}(0)$ is the ``asymptotic version" of $\tilde{X}_0^\trop$, given by 
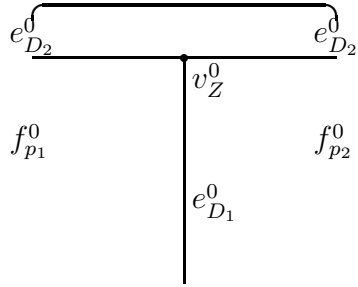
\begin{center}
\setlength{\unitlength}{1cm}
\begin{picture}(10,5)
\thicklines
\put(5,4){\circle*{0.1}}
\put(3,4){\line(1,0){2}}
\put(5,4){\line(1,0){2}}
\put(5,4){\line(0,-1){3}}
\put(5.1,3.6){$v_Z^0$}
\put(5.1,2){$e_{D_1}^0$}
\put(2.7,2.8){$f_{p_1}^0$}
\put(6.7,2.8){$f_{p_2}^0$}
\put(2.7,4.2){$e_{D_{2}}^0$}
\put(6.7,4.2){$e_{D_{2}}^0$}
\put(5,4.5){\oval(4,0.4)[t]}
\end{picture}
\end{center}
which can be identifed with the tropicalization of $Z$, viewed as a
log scheme for the divisorial log structure defined by $D=D_1 \cup D_2$.

Contact orders for stable log maps to a log scheme are integral points in the 
tropicalization of this log scheme 
(see Definition 2.3.12 of \cite{abramovich2017decomposition}).
We denote $\tilde{u}_p$ the contact order for stable log maps to $\tilde{X}_0$ defined by $\beta \cdot D_1$ times the integral generator (pointing away from $v_Z^0$) of the edge 
$e_{D_1}^0$ in $\Sigma(\tilde{\pi}_0)^{-1}(0)$.

\subsection{Invariants of the special fiber} \label{inv_special2}

Let $\overline{M}_g(\tilde{X}_0,\beta_Z)$
be the moduli space of genus $g$ class 
$\beta_Z$ stable log maps to 
$\tilde{\pi}_0 \colon \tilde{X}_0 \rightarrow \pt_\NN$, with a single point of contact order $\tilde{u}_p$. 
Let 
\[[\overline{M}_g(\tilde{X}_0,\beta_Z)]^\virt  \in A_g(\overline{M}_g(\tilde{X}_0,\beta_Z),\Q)\]
be its virtual fundamental class.
By deformation invariance of log Gromov-Witten invariants,
we have 
\[ N_{g,\beta}^{Z/D} = \int_{[\overline{M}_g(\tilde{X}_0,\beta_Z)]^\virt}
(-1)^g \lambda_g  \,.\]

\subsection{Tropical curves} \label{tropical_curves2}
Let
\begin{center}
\begin{tikzcd}
C \arrow{r}{f} \arrow{d}{\nu}
& \tilde{X}_0 \arrow{d}{\tilde{\pi}_0}\\
W \arrow{r}{g} & \pt_\NN \,,
\end{tikzcd}
\end{center}
be an element of $\overline{M}_g(\tilde{X}_0,\beta_Z)$.
Let 
\begin{center}
\begin{tikzcd}
\Sigma(C) \arrow{r}{\Sigma(f)} \arrow{d}{\Sigma(\nu)}
& \Sigma(\tilde{X}_0) \arrow{d}{\Sigma(\tilde{\pi}_0)}\\
\Sigma(W) \arrow{r}{\Sigma(g)} & \Sigma(\pt_\NN)
\end{tikzcd}
\end{center}
be its tropicalization.
For every $b \in \Sigma(g)^{-1}(1)$,
let 
\[ \Sigma(f)_b \colon
\Sigma(C)_b \rightarrow \Sigma(\tilde{\pi_0})^{-1}
(1) =\tilde{X}_0^\trop \]
be the fiber of $\Sigma(f)$ over $b$.

The definition of $\overline{M}_g(\tilde{X}_0,\beta_Z)$ fixes the set of unbounded edges of 
$\Sigma(C)_b$. 
The contact order $\tilde{u}_p$ at a single point is dual to an unbounded edge 
$E_0$ of $\Sigma(C)_b$. Furthermore, $\Sigma(C)_b$ has no other unbounded edge.
The contact order $\tilde{u}_p$
implies that $E_0$ is mapped by $\Sigma(C)_b$ onto an unbounded ray in 
$\tilde{X}_0^\trop$, going down in a way parallel to 
$e_{\PP^1_{p_1}}$, $e_{D_1}$ and $e_{\PP^1_{p_2}}$.
The fact that the image of $E_0$ by $\Sigma(C)_b$ can a priori
be any ray parallel to $e_{D_1}$ reflects the fact that 
a stable log map to $Y$ with a single point of maximal tangency along $D_1$
degenerates in a stable log map to $\tilde{X}_0$ with maximal tangency along $D_1 \cup \PP^1_{p_1} \cup \PP^1_{p_2}$,
and a priori not only along $D_1$. We denote $V_0$ the vertex of $\Gamma$ to which the 
unbounded edge $E_0$ is attached.

The tropical balancing condition has to be satisfied everywhere except at the vertices $v_Y$ 
and $v_{\tilde{\PP}}$ (the corresponding components do not have toric divisorial log structures).
At $v_{\tilde{\PP}}$, the tropical balancing condition is only modified in the 
horizontal direction: it still holds in the vertical direction.
We refer to Proposition 1.15 of \cite{MR3011419} for the general form of the
balancing conditions for stable log maps.

For $V$ a vertex of $\Sigma(C)_b$ and $E$ an edge adjacent to 
$V$, we denote $v_{V,E}$ the slope measured from $V$
of the restriction of $\Sigma(f)_b$ to $E$. Given our picture of 
$\tilde{X}_0^\trop$, we can identify $v_{V,E}$ with a well-defined element of $\Z^2$, except if 
$\Sigma(f)_b(E)$ is contained in $e_{D_{2,0}} \cup e_{D_{2,\infty}}$, in which case $v_{V,E}$
defines two elements of $\Z^2$ corresponding to the two copies of $e_{D_{2,0}}
\cup e_{D_{2,\infty}}$ and identified 
by the map gluing the two copies of $e_{D_{2,0}} \cup e_{D_{2,\infty}}$.

Using the same identification with $\Z^2$, we can write $\tilde{u}_p=(-\beta \cdot D_1,0)$, and we will use the standard scalar product of vectors in $\Z^2$. In particular, 
$v_{V,E} \cdot \tilde{u}_p=0$ is equivalent to $\Sigma(f)_b$ horizontal, 
$v_{V,E} \cdot \tilde{u}_p>0$ is equivalent to $\Sigma(f)_b$ pointing downward from $V$,
and $v_{V,E} \cdot \tilde{u}_p<0$ is equivalent to $\Sigma(f)_b$ pointing upward from $V$.

The following Proposition is the analogue for the present degeneration of 
the Proposition 
\ref{prop_trop_curves}
for the degeneration considered in Section \ref{section_deg_point}.

\begin{prop} \label{prop_trop_curves2}
Let
\begin{center}
\begin{tikzcd}
C \arrow{r}{f} \arrow{d}{\nu}
& \tilde{X}_{0} \arrow{d}{\tilde{\pi}_0}\\
W \arrow{r}{g} & \pt_\NN \,,
\end{tikzcd}
\end{center}
be an element of $\overline{M}_g(\tilde{X}_0,\beta_Z)$.
Let 
\begin{center}
\begin{tikzcd}
\Sigma(C) \arrow{r}{\Sigma(f)} \arrow{d}{\Sigma(\nu)}
& \Sigma(\tilde{X}_0) \arrow{d}{\Sigma(\tilde{\pi}_0)}\\
\Sigma(W) \arrow{r}{\Sigma(g)} & \Sigma(\pt_\NN) 
\end{tikzcd}
\end{center}
be its tropicalization.
Let $b$ be in the interior of $\Sigma(g)^{-1}(1)$.
Assume that $\Sigma(C)_b$ is a graph of genus $0$. Then:
\begin{itemize}
    \item For every $V$ vertex of $\Sigma(C)_b$, we have
    $\Sigma(f)_b(V)\in (e_{D_1} \cup e_{D_{2,0}})$.
    \item Let $V$ and $V'$ be vertices of $\Sigma(C)_b$ such that 
    $\Sigma(f)_b(V)$ is contained in the interior of $e_{D_1}$
    and $\Sigma(f)_b(V')$ is contained in the interior of $e_{D_{2,0}}$.
    Then, there is no edge of $\Sigma(C)_b$ connecting $V$ and $V'$.
    \item We have $E_0 \subset e_{D_1}$.
\end{itemize}
\end{prop}

The proof of Proposition \ref{prop_trop_curves2} is parallel to the proof of 
Proposition \ref{prop_trop_curves}.
We first remark that the natural analogues of Lemmas \ref{lem_descent},
\ref{lem_V_0_lower}, \ref{lem_D_2} still hold, with the same proofs.

\begin{lem} \label{lem_D_2_infinity}
There is no vertex 
$V$ of $\Sigma(C)_b$ such that $\Sigma(f)_b(V) \in e_{D_{2,\infty}}$
and $\Sigma(f)_b(V) \neq v_{\tilde{\PP}}$.
\end{lem}

\begin{proof}
Let $V$ be a vertex of $\Sigma(C)_b$ such that $\Sigma(f)_b(V) \in e_{D_{2,\infty}}$
and $\Sigma(f)_b(V) \neq v_{\tilde{\PP}}$. Arguing as in the proof of 
Lemma \ref{lem_D_2}, we get that any edge $E$ adjacent to $V$ should be horizontal.
Look at $V$ the leftmost (or rightmost, in our 
picture of $\tilde{X}_0^\trop$) of such vertices. Then, using the balancing condition,
there should be an unbounded edge adjacent to $V$ and whose image by 
$\Sigma(f)_b$ is contained in $e_{D_{2,\infty}}$. This contradicts the fact that 
$E_0$ is the only unbounded edge of $\Sigma(C)_b$.
\end{proof}

Recall that $\tilde{\PP}$ is obtained from the projective bundle 
$\PP \rightarrow D_2$ by blowing-up the $\beta \cdot D_2$ points $s_j(0)  \in D_{2,\infty}$,
$1 \leqslant j \leqslant \beta \cdot D_2$. Let $F_j$, $1\leqslant j \leqslant \beta \cdot D_2$, be the exceptional divisor and let $\PP^1_j$, $1\leqslant j \leqslant \beta \cdot D_2$, be the strict transforms of the $\PP^1$ fibers of $\PP \rightarrow D_2$ passing through $s_j(0)$.
In particular, for every $1 \leqslant j \leqslant \beta \cdot D_2$, $\PP^1_j$
and $F_j$ intersect in one point. We denote
$\tilde{\PP} \rightarrow D_2$ the composition of the blow-up morphism 
$\tilde{\PP} \rightarrow \PP$ with the $\PP^1$-fibration 
$\PP \rightarrow D_2$.

\begin{lem} \label{lem_no_h_component2}
Let $C'$ be an irreducible component of $C$ such that $f(C') \subset \tilde{\PP}$.
Then $f(C')$ is contained in a fiber of $\tilde{\PP} \rightarrow D_2$.
\end{lem}

\begin{proof}
Given Lemma \ref{lem_D_2_infinity}, the analogue of the 
proof of Lemma \ref{lem_no_h_component} applies.
\end{proof}

\begin{lem} \label{lem_P_tilde}
Let $C'$ be an irreducible component of $C$ such that $f(C') \subset \tilde{\PP}$.
Then, either there exists $j$, $1 \leqslant j \leqslant \beta \cdot D_2$, such that 
$f(C') \subset \PP^1_j$, or $f(C') \subset \PP^1_{p_1}$ or $f(C') \subset \PP^1_{p_2}$.
\end{lem}

\begin{proof}
By Lemma \ref{lem_no_h_component2}, we already know that $f(C')$ is contained in a fiber
of $\tilde{\PP} \rightarrow D_2$. To conclude, it is enough to show that
$f(C') \cap D_{2,\infty}$ is included in $(D_{2,\infty} \cap \PP^1_{p_1})\cup 
(D_{2,\infty} \cap \PP^1_{p_2})$. If it were not the case, then there would exist an edge of $\Sigma(C)_b$ which is mapped non-trivially in $e_{D_{2,\infty}}$,
in contradiction with the proof of Lemma \ref{lem_D_2_infinity}.
\end{proof}

Let $\Gamma_{\tilde{\PP}}$ be the subgraph of $\Sigma(C)_b$
consisting of vertices mapped by $\Sigma(f)_b$
on $e_{\PP^1_{p_1}} \cup e_{\PP^1_{p_2}}$.
Let $\tau$ be a connected component of $\Gamma_{\tilde{\PP}}$
and let $C_\tau$ be the curve dual to $\Gamma_\tau$. 
By Lemma \ref{lem_no_h_component2}, $f(C_\tau)$ is contained 
in a fiber of $\tilde{\PP} \rightarrow D_{2,0}$.
In particular, as $C$ is connected, we have $d_\tau \coloneq
[f(C_\tau)] \cdot D_{2,0} \geqslant 1$. In fact, 
$d_\tau$ is the sum of weights of edges of $\Sigma(C)_b$
adjacent to $\tau$ and mapped non-trivially in $e_{D_{2,0}}$.

It follows from Lemma \ref{lem_P_tilde} and from $[f(C)]=\beta_Z$
(in particular $[f(C)]\cdot [F_j]=1$),
that,
for every
$1 \leqslant j \leqslant \beta \cdot D_2$, 
the vertices of $\Sigma(C)_b$ dual to components $C'$ such that $f(C') \subset \PP^1_j$
define a connected component $\tau_j$ of $\Gamma_{\tilde{\PP}}$
with $d_{\tau_j} =1$.

Let $\cV_Y$ be the set of vertices $V$ of
$\Sigma(C)_b$ such that 
$\Sigma(f)_b(V)=v_Y$, and which are connected to
some vertex in some connected component 
of $\Gamma_{\tilde{\PP}}$ by some path mapped
in $e_{D_{2,0}}$ by $\Sigma(f)_b$.
Let $C_Y$ be the union of irreducible components of $C$ dual to vertices 
in $\cV_Y$. Denote 
$\beta_Y \coloneqq [f(C_Y)]
\in H_2(Y,\Z)$.

\begin{lem} \label{lem_homology2}
We have $\beta_Y \cdot D_{2,0}=\beta \cdot D_{2,0}$, and every connected component of $\Gamma_{\tilde{\PP}}$ is of the form
$\tau_j$ for some $1 \leqslant j \leqslant \beta \cdot D_{2,0}$. 
In particular, no vertex of $\Sigma(C)_b$ is mapped by $\Sigma(f)_b$
in the interior of $e_{\PP^1_{p_1}}$ or the 
interior of $e_{\PP^1_{p_2}}$.
\end{lem}

\begin{proof}
This is parallel to the proof of Lemma \ref{lem_homology}.
It follows from the analogue of Lemma \ref{lem_D_2} and from the horizontal balancing condition, satisfied on $e_{D_{2,0}}$ away from $v_Y$ and $v_{\tilde{\PP}}$, that 
\[\beta_Y \cdot D_2 \geqslant \sum_\tau d_\tau \,,\]
where the sum is over the connected components of $\Gamma_\PP$.

By Lemma \ref{lem_no_h_component2}, no irreducible component of $C$ is mapped onto $D_2$,
so the image of every irreducible
irreducible component of $C$ intersects 
$D_2$ non-negatively and so
$\beta \cdot D_2 \geqslant \beta_Y\cdot D_2$. On the other hand, we have 
$\sum_{\tau} d_\tau
\geqslant \sum_{j=1}^{\beta \cdot D_2}
d_{\tau_j} \geqslant \beta \cdot D_2$.

The combination of the previous inequalities gives 
\[\beta \cdot D_2 \geqslant \beta_Y\cdot D_2
\geqslant
\sum_{\tau} d_\tau
\geqslant \sum_{j=1}^{\beta \cdot D_2}
d_{\tau_j} \geqslant \beta \cdot D_2\,,\]
and so all these inequalities are in fact equalities. Lemma 
\ref{lem_homology2} follows.
\end{proof}

\begin{lem} \label{lem_no_D_12}
Let $C'$ be an irreducible component of $C$ such that $f(C') \subset Y$ and $C'$
is not contracted by $f$. Then 
\[(f(C') \cap D_2) \subset \{\PP^1_j \cap D_{2,0}, 1 \leqslant j \leqslant \beta \cdot D_2\} \,.\]
\end{lem}

\begin{proof} 
By Lemma \ref{lem_no_h_component2}, no irreducible component of $C$ is mapped onto $D_2$,
so the image of every irreducible
irreducible component of $C$ intersects 
$D_2$ non-negatively.
But according to Lemma
\ref{lem_homology2}, the total intersection number $\beta \cdot D_2$ of images by $f$ components of $C$ mapping to $Y$ with $D_2$
is already accounted by 
the intersection points $\PP_j^1 \cap D_2, 
1 \leqslant j \leqslant \beta \cdot D_2$.
\end{proof}

\begin{lem} \label{lem_no_f2}
There is no vertex $V$ of $\Sigma(C)_b$
such that $\Sigma(f)_b(V)$ is contained in the interior of 
either $f_{p_1,0}$, $f_{p_2,0}$, $f_{p_1,\infty}$ or $f_{p_2,\infty}$.
\end{lem}

\begin{proof}
It follows from Lemma \ref{lem_homology2} and Lemma
\ref{lem_no_D_12} that no non-contracted component of $C$
has an image by $f$ intersecting $p_1$, $p_2$, 
$\PP^1_{p_1} \cap D_{2,\infty}$ or $\PP^1_{p_2} \cap D_{2,\infty}$.
As $C$ is connected, this implies that no component of $C$
can be contracted by $f$ on $p_1$, $p_2$, $\PP^1_{p_1} \cap D_{2,\infty}$ or $\PP^1_{p_2} \cap D_{2,\infty}$.
\end{proof}

The first two points of Proposition 
\ref{prop_trop_curves2} follow from 
Lemma \ref{lem_homology2}, Lemma \ref{lem_no_D_12}  and
Lemma \ref{lem_no_f2}.
In particular, we have 
$\Sigma(f)_b(V_0) \in e_{D_1} \cup e_{D_{2,0}}$. As $\beta \cdot D_1 > 0$, there is at least one non-contracted component of
$C$ mapping to $Y$ and intersecting 
$D_1$, and so there is at least one vertex $V$ and an edge $E$ such that 
$v_{V,E}\cdot u_p>0$. Using the analogue of Lemma 
\ref{lem_V_0_lower}, it follows that 
$\Sigma(f)_b(V_0) \in e_{D_{2,0}}$ is 
impossible, and so $\Sigma(f)_b(V_0) \in e_{D_1}$, proving the third point of Proposition \ref{prop_trop_curves2}.

\subsection{End of the proof of Proposition \ref{prop_exchange_blow_ups}}
\label{end_proof2}

Given Proposition \ref{prop_trop_curves2}, the rest of the proof of 
Proposition \ref{prop_exchange_blow_ups} is parallel to the proof of Proposition of 
Proposition \ref{prop_deg_point}. Proposition \ref{prop_trop_curves2}
gives a way to identify the rigid decorated parametrized tropical curves contributing to the degeneration formula after integration of $\lambda_g$: they are the ``obvious ones" from the point of view of the usual degeneration formula in relative Gromov-Witten theory, with $V_0$ a vertex mapping to $v_Y$ and $V_j$ vertices mapping to $v_{\tilde{\PP}}$,
for $1 \leqslant j \leqslant \beta \cdot D_2$. The factors 
$\frac{1}{2 \sin \left( \frac{\hbar}{2}\right)}$
are the contributions of the vertices $V_j$, $1 \leqslant j \leqslant \beta \cdot D_2$, and comes from relative Gromov-Witten theory of $\PP^1_j \simeq \PP^1$. Indeed, let $M$ be the moduli space of degree $1$ stable maps to 
$\PP^1$ relative to a point $\infty \in \PP^1$, $\nu \colon C \rightarrow M$ the universal curve, $f \colon C \rightarrow \PP^1$ the universal stable map, and  $[M]^{\virt}$ the natural virtual fundamental class. As the normal bundle  of $\PP^1_j$ in $ \tilde{\PP}$ is $\cO(-1)$,
the virtual fundamental class on $M$ coming from the surface $\tilde{\PP}$ is obtained by intersecting $[M]^{\virt}$ with 
$e(R^1 \nu_{*}f^{*}\cO(-1))$. Therefore, 
the contribution of $V_j$ is $\sum_{g \geq 1} \left( \int_{[M]^{\virt}} (-1)^g \lambda_g \, e(R^1 \nu_{*}f^{*}\cO(-1)) \right) \hbar^{2g-1}$, which is equal to 
$\frac{1}{2 \sin \left( \frac{\hbar}{2}\right)}$ by 
\cite{MR2115262} (see the proof of Theorem 5.1).

\section{End of the proof of Theorem \ref{main_thm_precise}}
\label{section_end_proof}

By combination of Proposition 
\ref{prop_deg_point} and Proposition \ref{prop_exchange_blow_ups}, we have 
\[N_{g,\beta}^{Z/D}=N_{g,\beta}^{Y/D_1}\,,\]
result stated as Theorem \ref{thm_gw_0} in the Introduction.

If $\mathbf{d}(\beta) \in \NN^{(Q^{Y/D_1}_\beta)_0}$, then, according to Theorem 8.13 of 
\cite{bousseau2018quantum_tropical}, we have 
\[ \Omega_\beta^{Y/D_1}(q^{\frac{1}{2}})
=(-1)^{\beta \cdot D_1+1} 
\left( 2 \sin \left( \frac{\hbar}{2}
\right) \right) \
\left( \sum_{g\geqslant 0} N_{g,\beta}^{Z/D} 
\hbar^{2g-1} \right) \,,\]
where $q=e^{i \hbar}$.
Indeed, as $\beta_Z \cdot E_j=1$, for all 
$1 \leqslant j \leqslant \beta \cdot D_2$, the class $\beta_Z$ is primitive in 
$H_2(Z,\Z)$ and so the integral log BPS invariant appearing in Theorem 8.13 of \cite{bousseau2018quantum_tropical}
coincides with the rational log BPS invariant defined in Conjecture 8.1 of 
\cite{bousseau2018quantum_tropical}.
Replacing $N_{g,\beta}^{Z/D}$ by 
$N_{g,\beta}^{Y/D_1}$ in the above formula for 
$\Omega_\beta^{Y/D_1}(q^{\frac{1}{2}})$, we get exactly the statement of Theorem
\ref{main_thm_precise}.

If $\mathbf{d}(\beta) \notin \NN^{(Q^{Y/D_1}_\beta)_0}$, it follows from the 
proof of Proposition 33 of 
\cite{bousseau2018quantum} that $N_{g,\beta}^{Z/D}=0$, for every $g \geqslant 0$.
By definition, we also have $\Omega_\beta^{Y/D}(q^{\frac{1}{2}})=0$
and so Theorem \ref{main_thm_precise} is still valid in this case.

\section{Examples}
\label{section_examples}

In this Section, we consider various examples of $Y$, $D_1$, $D_2$ as in Section
\ref{gw_side}, and we construct a corresponding quiver following the recipe of Section \ref{quiver_construction}.
Each time this quiver is acyclic, Theorem 
\ref{main_thm_precise} applies. 
To index our examples, we use the notation
$Y(D_1^2,D_2^2)$. In total, we give $7$ examples with an acyclic quiver, and  $3$ examples with a non-acyclic quiver. 
The reader is invited to try other examples.

According to Section \ref{quiver_construction}, the main step
in the quiver construction is the identification of a 
diagram 
\begin{center}
\begin{tikzcd}
& (\tilde{Z},\tilde{D}) \arrow{dl}[swap]{\pi_Z} \arrow{dr}{\pi_{\bar{Z}}}\\
(Z,D) & & (\bar{Z},\bar{D}) \,,
\end{tikzcd}
\end{center}
with $(\bar{Z},\bar{D})$  toric,
$\pi_Z$ is a sequence of corner blow-ups of 
$(Z,D)$, and $\pi_{\bar{Z}}$ is a sequence of 
interior blow-ups of $(\bar{Z},\bar{D})$.

For practical computations, it is useful to know that, given a log Calabi-Yau surface with maximal boundary $(\bar{Z},\bar{D})$, if the sequence of self intersection numbers of irreducible components of $D$ can be realized as the sequence of self intersection numbers of toric divisors of a toric surface, then in fact $(\bar{Z},\bar{D})$ is toric. See Lemma 2.10 of \cite{friedman2015geometry} for a proof of this result. Thus, the identification of a diagram as above mostly requires only numerical work.

\subsection{$\PP^2(1,4)$} \label{ex_p2}
We consider $Y=\PP^2$, $D_1$ a line and 
$D_2$ a smooth conic not tangent to $D_1$,
so that $D_1^2=1$ and $D_2^2=4$.
We first explain how to construct a toric model of $(Y,D)$, where $D=D_1 \cup D_2$.
Denote $p_1$ and $p_2$ the two intersection 
points of $D_1$ and $D_2$.
Denote $L$ the line tangent to $D_2$ at the point $p_1$.

We first blow-up $p_1$. Let $F_1$ be the exceptional divisor and let $p_1'$ be the intersection point of $F_1$ and $D_2$. 
The line $L$ passes through $p_1'$. 
We then blow-up $p_1'$. Let $F_2$ be the exceptional divisor. The line $L$ is now a (-1)-curve, which can be contracted. 
We get a log Calabi-Yau surface, with boundary $D_1 \cup F_1 \cup F_2 \cup D_2$,
where $D_1^2=0$, $F_1^2=-2$, $F_2^2=0$, 
$D_2^2=2$, that is, a Hirzebruch surface 
$\F_2$.
The fan of $\F_2$ is given by four rays generated by 
$(-1,0)$, $(0,-1)$, $(0,1)$ and $(1,2)$: 

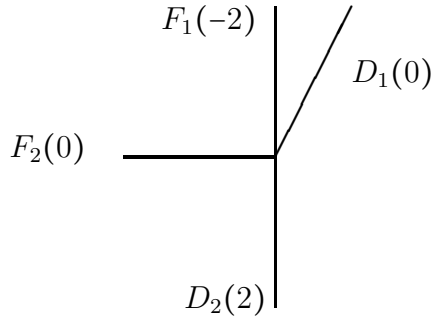
\begin{center}
\setlength{\unitlength}{1cm}
\begin{picture}(10,5)
\thicklines
\put(3.8,0.5){$D_2(2)$}
\put(3.5,4.2){$F_1(-2)$}
\put(1.5,2.5){$F_2(0)$}
\put(6,3.5){$D_1(0)$}
\put(5,2.5){\line(-1, 0){2}}
\put(5,2.5){\line(0,-1){2}}
\put(5,2.5){\line(0,1){2}}
\put(5,2.5){\line(1,2){1}}
\end{picture}
\end{center}
We wrote near each ray the corresponding divisor and in parenthesis it self intersection number.

Let $\beta=d \in \Z=H_2(\PP^2,\Z)$. 
Following the class $\beta$ through the previous blow-ups and blow-down, we get the class $\beta_{\F_2} \in H_2(\F_2,\Z)$ such that $\beta_{\F_2}\cdot F_2=d$, 
$\beta_{\F_2}\cdot F_1=0$, $\beta_{\F_2}\cdot D_1=d$ and $\beta_{\F_2}\cdot D_2=2d$.
Remark that the balancing collection
\[ d(-1,0)+2d(0,-1)+d(1,2)=0  \]
is indeed satisfied.

The surface $Z$ is obtained by 
blowing-up $\beta \cdot D_2=2d$ points on $D_2$.
It follows that we get 
$\tilde{Z}$ starting from $\F_2$ and blowing-up $2d$ points on the divisor dual to the ray of direction $(0,-1)$ and one point on the divisor dual to the ray of direction $(-1,0)$.
As 
\[\langle (-1,0) , (0,-1) \rangle =1\,,\]
it follows from Section \ref{quiver_construction} that the corresponding quiver and dimension vectors are given by 
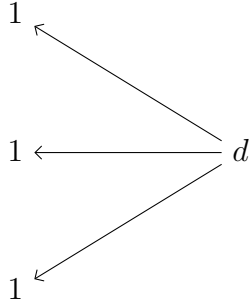
\begin{center}
\begin{tikzpicture}[>=angle 90]
\matrix(a)[matrix of math nodes,
row sep=3em, column sep=3em,
text height=1.5ex, text depth=0.25ex]
{1& &\\
1& &d\\
1& &\\};
\path[<-](a-1-1) edge (a-2-3);
\path[<-](a-2-1) edge (a-2-3);
\path[<-](a-3-1) edge (a-2-3);
\end{tikzpicture}
\end{center}
with $2d$ vertices of dimension one on the left. This quiver is acyclic and so Theorem \ref{main_thm_precise} applies. 
This quiver coincides with the quiver considered in
\cite{reineke2018moduli}.

\subsection{$\PP^2(4,1)$}
We consider $Y=\PP^2$, $D_1$ a smooth conic and $D_2$ a line not tangent to $D_1$.
Let $\beta=d \in \Z=H_2(\PP^2,\Z)$.
The surface $Z$ is obtained by blowing-up
$\beta \cdot D_2=d$ points on $D_2$. 
It follows from the above discussion of
$\PP^2(1,4)$ that we get $\tilde{Z}$ starting from $\F_2$ and blowing-up $d$ points on the divisor dual to the ray of direction $(1,2)$ and one point of the ray of direction $(-1,0)$.
As 
\[\langle (1,2), (-1,0) \rangle =2 \,,\]
it follows that the corresponding quiver and dimension vectors are given by 
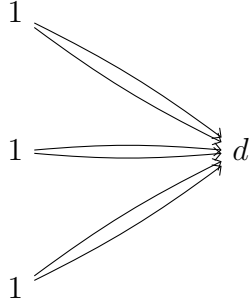
\begin{center}
\begin{tikzpicture}[>=angle 90]
\matrix(a)[matrix of math nodes,
row sep=3em, column sep=3em,
text height=1.5ex, text depth=0.25ex]
{1& &\\
1& &d\\
1& &\\};
\path[->][bend left=5](a-1-1) edge  (a-2-3);
\path[->][bend right=5](a-1-1) edge  (a-2-3);
\path[->][bend left=5](a-2-1) edge (a-2-3);
\path[->][bend right=5](a-2-1) edge (a-2-3);
\path[->][bend left=5](a-3-1) edge  (a-2-3);
\path[->][bend right=5](a-3-1) edge  (a-2-3);
\end{tikzpicture}
\end{center}
with $d$ vertices of dimension one on the left.
This quiver is acyclic and so Theorem \ref{main_thm_precise} applies.

\subsection{$\F_0 (2,2)$}
We consider
$Y=\F_0=\PP^1 \times \PP^1$, and $D_1$, $D_2$ smooth divisors of 
degree $(1,1)$. We have $D_1^2=D_2^2=2$.


We first explain how to construct a toric model for $(Y,D)$, where $D=D_1 \cup D_2$.
We blow-up one of the two intersection points of $L_1$ and $L_2$.
Let $F$ be the exceptional divisor. The 
strict transforms of the horizontal (degree $(1,0)$) and vertical
(degree $(0,1)$) $\PP^1$s passing through this point are disjoint $(-1)$-curves, which can be contracted. The resulting log Calabi-Yau surface is $\PP^2$
with its toric boundary.
The intersection numbers of $\beta$ with these two interior $(-1)$-curves
are $d_1$ and $d_2$.
The fan of $\PP^2$ is fiven by three rays generated by 
$(-1,0)$, $(0,-1)$, $(1,1)$. 

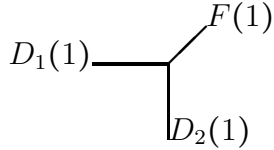
\begin{center}
\setlength{\unitlength}{1cm}
\begin{picture}(10,3)
\thicklines
\put(5,1.5){\line(-1, 0){1}}
\put(5,1.5){\line(0,-1){1}}
\put(5,1.5){\line(1,1){0.5}}
\put(5,0.5){$D_2(1)$}
\put(2.9,1.5){$D_1(1)$}
\put(5.5,2){$F(1)$}
\end{picture}
\end{center}
We wrote near each ray the corresponding divisor and in parenthesis its self intersection number.

Let $\beta=(d_1,d_2) \in \Z^2=H_2(\F_0,\Z)$. We have $\beta \cdot D=2(d_1+d_2)$ and 
$\beta \cdot D_1=\beta \cdot D_2=d_1+d_2$.
Following the class $\beta$ through the previous blow-ups and blow-down, we get the class $\beta_{\PP^2} \in H_2(\PP^2,\Z)$ such that $\beta_{\PP^2}\cdot D_1=\beta_{\PP^2}\cdot D_2=d_1+d_2$ and 
$\beta_{\PP^2}\cdot F=d_1+d_2$.
Remark that the balancing condition
\[(d_1+d_2)(-1,0)+(d_1+d_2)(0,-1)+(d_1+d_2)(1,1)=0\]
is indeed satisfied.

The surface $\tilde{Z}$ is obtained by blowing-up 
$d_1+d_2$ points on the divisor dual to the ray of direction $(0,-1)$ and two points on the divisor dual to the ray of direction $(1,1)$.
As 
\[\langle (0,-1),(1,1) \rangle =1\,,\]
it follows from Section
\ref{quiver_construction} that the corresponding quiver and dimension vectors are given by 
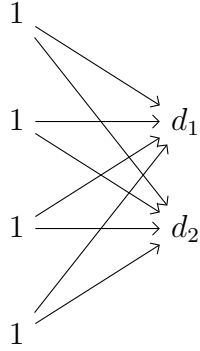
\begin{center}
\begin{tikzpicture}[>=angle 90]
\matrix(a)[matrix of math nodes,
row sep=2em, column sep=2em,
text height=1.5ex, text depth=0.25ex]
{1& &\\
1& &d_1\\
1& &d_2\\
1& &\\};
\path[->](a-1-1) edge  (a-2-3);
\path[->](a-2-1) edge  (a-2-3);
\path[->](a-3-1) edge  (a-2-3);
\path[->](a-4-1) edge  (a-2-3);
\path[->](a-1-1) edge  (a-3-3);
\path[->](a-2-1) edge  (a-3-3);
\path[->](a-3-1) edge  (a-3-3);
\path[->](a-4-1) edge  (a-3-3);
\end{tikzpicture}
\end{center}
with $d_1+d_2$ vertices of dimension $1$
on the left.
This quiver is acyclic and so Theorem \ref{main_thm_precise} applies. 

\subsection{$\F_1(0,4)$}
We consider $Y=\F_1$.  We denote
$C_{-1}$, $C_1$, $f_0$, $f_\infty$
the toric divisors of $\F_1$ so that $C_{-1}^2=-1$,
$C_1^2=1$, $f_0^2=f_\infty^2=0$.
We have the linear equivalence relations $f_0 \sim f_\infty$ and 
$C_1 \sim C_{-1}+f_0$.

We take
$D_1=f_0$ and $D_2$
a smooth rational curve linearly equivalent to $C_{-1}+f_\infty+C_1 \sim
C_{-1}+2f_0$, and intersecting $D_1$ transversally. We have $D_1^2=0$ and $D_2^2=4$.


We first explain how to construct a toric model for $(Y,D)$, where $D=D_1 \cup D_2$. We remark that $C_{-1}$ is a 
$(-1)$-curve disjoint from $D_2$ and so can be contracted. The resulting log Calabi-Yau surface is 
$\PP^2$, with $D_1$ being a line and $D_2$ being a conic.
So it is enough to take the known toric model for 
$(\PP^2, D_1 \cup D_2)$, which is given by $\F_2$ as described in Section
\ref{ex_p2}.

The fan of $\F_2$ is given by four rays generated by 
$(-1,0)$, $(0,-1)$, $(0,1)$ and $(1,2)$: 
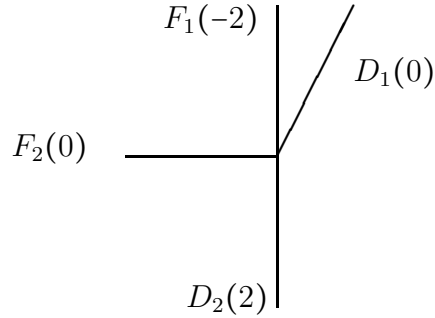
\begin{center}
\setlength{\unitlength}{1cm}
\begin{picture}(10,5)
\thicklines
\put(3.8,0.5){$D_2(2)$}
\put(3.5,4.2){$F_1(-2)$}
\put(1.5,2.5){$F_2(0)$}
\put(6,3.5){$D_1(0)$}
\put(5,2.5){\line(-1, 0){2}}
\put(5,2.5){\line(0,-1){2}}
\put(5,2.5){\line(0,1){2}}
\put(5,2.5){\line(1,2){1}}
\end{picture}
\end{center}
We wrote near each ray the corresponding divisor and in parenthesis its self intersection number.

Let $\beta=d_1 [C_{-1}]+d_2 [f_0]
\in H_2(\F_1,\Z)$. We have $\beta \cdot D=d_1+2d_2$,
$\beta \cdot D_1=d_1$, $\beta \cdot D_2=2d_2$.
Following the class $\beta$ through 
the previous blow-ups and blow-down, we get the class $\beta_{\F_2} \in H_2(\F_2,\Z)$
such that 
$\beta_{\F_2}\cdot D_1=d_1+\beta \cdot C_{-1}
=d_1+(d_2-d_1)=d_2$, 
$\beta_{\F_2}\cdot D_2=2d_2$,
$\beta_{\F_2}\cdot F_2=d_2$,
$\beta_{\F_2}\cdot F_1=0$.
Remark that the balancing condition
\[d_2(-1,0)+2d_2(0,-1)+d_2(1,2)=0\]
is indeed satisfied.

The surface $\tilde{Z}$ is obtained from 
$\F_2$ by blowing-up $2d_2$ points on the divisor dual to the ray of direction
$(0,-1)$, one point on the divisor dual to the ray of direction
$(-1,0)$, and one point on the divisor dual to the ray of direction $(1,2)$.
As 
\[\langle (-1,0),(0,-1)\rangle=1 \,,
\langle (0,-1),(1,2)\rangle=1\,,
\langle (1,2),(-1,0)\rangle =2 \,,\]
it follows from Section \ref{quiver_construction} that the corresponding quiver and dimension vectors are given by
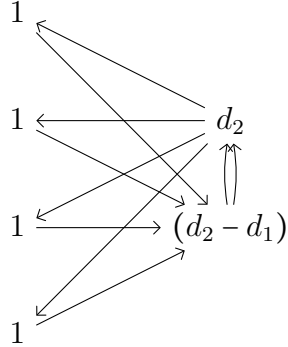
\begin{center}
\begin{tikzpicture}[>=angle 90]
\matrix(a)[matrix of math nodes,
row sep=2em, column sep=2em,
text height=1.5ex, text depth=0.25ex]
{1& &\\
1& &d_2\\
1& &(d_2-d_1)\\
1& &\\};
\path[<-](a-1-1) edge  (a-2-3);
\path[<-](a-2-1) edge  (a-2-3);
\path[<-](a-3-1) edge  (a-2-3);
\path[<-](a-4-1) edge  (a-2-3);
\path[->](a-1-1) edge  (a-3-3);
\path[->](a-2-1) edge  (a-3-3);
\path[->](a-3-1) edge  (a-3-3);
\path[->](a-4-1) edge  (a-3-3);
\path[<-][bend left=8](a-2-3) edge (a-3-3);
\path[<-][bend right=8](a-2-3) edge (a-3-3);
\end{tikzpicture}
\end{center}
with $2d_2$ vertices of dimension $1$ on the left.
This quiver contains oriented cycles and so Theorem \ref{main_thm_precise} does not apply.

\subsection{$\F_1(4,0)$}
We consider $Y=\F_1$.  We denote
$C_{-1}$, $C_1$, $f_0$, $f_\infty$
the toric divisors of $\F_1$ so that $C_{-1}^2=-1$,
$C_1^2=1$, $f_0^2=f_\infty^2=0$.
We have the linear equivalence relations $f_0 \sim f_\infty$ and 
$C_1 \sim C_{-1}+f_0$.

We take
$D_2=f_0$ and $D_1$
a smooth rational curve linearly equivalent to $C_{-1}+f_\infty+C_1 \sim
C_{-1}+2f_0$, and intersecting $D_2$ transversally. We have $D_2^2=0$ and $D_1^2=4$.

It follows from the above discussion of 
$\F_1(0,4)$ that $\tilde{Z}$
is obtained from 
$\F_2$ by blowing-up $d_1+1$ points on the divisor dual to the ray of direction 
$(1,2)$, one point on the divisor dual to the ray of direction $(-1,0)$. 

As 
\[\langle (1,2),(-1,0)\rangle =2 \,,\]
it follows from Section \ref{quiver_construction} that the corresponding quiver and dimension vectors are given by
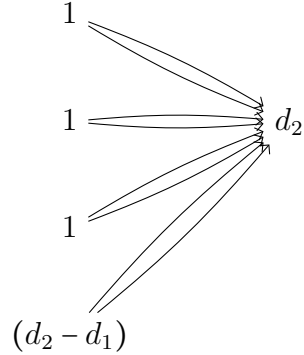
\begin{center}
\begin{tikzpicture}[>=angle 90]
\matrix(a)[matrix of math nodes,
row sep=2em, column sep=2em,
text height=1.5ex, text depth=0.25ex]
{1& &\\
1& &d_2\\
1& &\\
(d_2-d_1)& & \\};
\path[->][bend left=5](a-1-1) edge (a-2-3);
\path[->][bend right=5](a-1-1) edge (a-2-3);
\path[->][bend left=5](a-2-1) edge (a-2-3);
\path[->][bend right=5](a-2-1) edge (a-2-3);
\path[->][bend left=5](a-3-1) edge (a-2-3);
\path[->][bend right=5](a-3-1) edge (a-2-3);
\path[->][bend left=5](a-4-1) edge (a-2-3);
\path[->][bend right=5](a-4-1) edge (a-2-3);
\end{tikzpicture}
\end{center}
with $d_1$ vertices of dimension $1$ on the left.
This quiver is acyclic and so Theorem
\ref{main_thm_precise} applies.
Remark that for 
$d_1=d_2$, this quiver reduces to the quiver describing the 
$\PP^2(4,1)$ example, that is, 
\ $\PP^2$ relatively to a conic.
Indeed, in this case, $\beta \cdot C_{-1}=0$, and the geometry can be directly 
reduced from $\F_1$ to $\PP^2$.

\subsection{$\F_1(1,3)$}
We consider $Y=\F_1$.  We denote
$C_{-1}$, $C_1$, $f_0$, $f_\infty$
the toric divisors of $\F_1$ so that $C_{-1}^2=-1$,
$C_1^2=1$, $f_0^2=f_\infty^2=0$.
We have the linear equivalence relations $f_0 \sim f_\infty$ and 
$C_1 \sim C_{-1}+f_0$.

We take $D_1=C_1$ and $D_2$ 
a smooth rational curve linearly equivalent to $C_{-1}+f_0+f_\infty
\sim C_{-1}+2f_0$. We have $D_1^2=1$
and $D_2^2=3$.


We first explain how to construct a toric model for $(Y,D)$, where $D=D_1 \cup D_2$. We blow-up the two intersection 
points of $D_1$ and $D_2$. 
Let $F_1$ and $F_2$ be the two exceptional divisors.
The strict transforms of the $\PP^1$
fibers (of degree $[f_0]=[f_\infty]$) passing through these points are disjoint 
$(-1)$-curves, which can be contracted. The resulting log Calabi-Yau surface is $\F_1$
with its toric boundary.

The fan of $\F_1$ is given by the four rays generated by 
$(-1,0)$, $(0,-1)$, $(0,1)$ and $(1,1)$: 
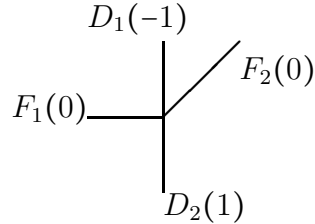
\begin{center}
\setlength{\unitlength}{1cm}
\begin{picture}(10,3)
\thicklines
\put(5,1.5){\line(-1, 0){1}}
\put(5,1.5){\line(0,-1){1}}
\put(5,1.5){\line(0,1){1}}
\put(5,1.5){\line(1,1){1}}
\put(3,1.5){$F_1(0)$}
\put(6,2){$F_2(0)$}
\put(4,2.7){$D_1(-1)$}
\put(5,0.2){$D_2(1)$}
\end{picture}
\end{center}
We wrote near each ray the corresponding divisor and in parenthesis its self intersection number.

Let $\beta=d_1 [C_{-1}]+d_2 [f_0] \in H_2(\F_1,\Z)$. We have 
$\beta \cdot D=d_1+2d_2$, $\beta \cdot D_1=d_2$, $\beta \cdot D_2=d_1+d_2$.
Following the class $\beta$ through the previous blow-ups and blow-downs, we get the class $\beta_{\F_1} \in H_2(\F_1,\Z)$
such that $\beta_{\F_1}\cdot D_1=d_2$,
$\beta_{\F_1}\cdot D_2=d_1+d_2$, 
$\beta_{\F_1}\cdot F_1=\beta_{\F_1}\cdot F_2
=\beta\cdot f_0=d_1$.
Remark that the balancing condition
\[d_2(0,1)+(d_1+d_2)(0,-1)+d_1(-1,0)+d_1(1,1)=0\]
is indeed satisfied.

The surface $\tilde{Z}$ is obtained from 
$\F_1$ by blowing-up $d_1+d_2$ points on the divisor dual to the ray of direction 
$(0,-1)$, one point on the divisor dual to the ray of direction $(-1,0)$, and one point on the divisor dual to the 
ray of direction $(1,1)$.
As
\[ \langle (-1,0),(0,-1)\rangle =1 \,,
\langle (1,1),(-1,0)\rangle =1 \,,
\langle (0,-1),(1,1) \rangle =1 \,, \] 
it follows from Section \ref{quiver_construction}
that the corresponding quiver and dimension vectors are given by
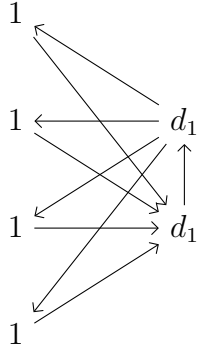
\begin{center}
\begin{tikzpicture}[>=angle 90]
\matrix(a)[matrix of math nodes,
row sep=2em, column sep=2em,
text height=1.5ex, text depth=0.25ex]
{1& &\\
1& &d_1\\
1& &d_1\\
1& &\\};
\path[<-](a-1-1) edge  (a-2-3);
\path[<-](a-2-1) edge  (a-2-3);
\path[<-](a-3-1) edge  (a-2-3);
\path[<-](a-4-1) edge  (a-2-3);
\path[->](a-1-1) edge  (a-3-3);
\path[->](a-2-1) edge  (a-3-3);
\path[->](a-3-1) edge  (a-3-3);
\path[->](a-4-1) edge  (a-3-3);
\path[<-](a-2-3) edge (a-3-3);
\end{tikzpicture}
\end{center}
with $d_1+d_2$ vertices of dimension $1$ on the left. This quiver 
contains oriented cycles
and so Theorem \ref{main_thm_precise} does not apply. 

\subsection{$\F_1(3,1)$}
We consider $Y=\F_1$.  We denote
$C_{-1}$, $C_1$, $f_0$, $f_\infty$
the toric divisors of $\F_1$ so that $C_{-1}^2=-1$,
$C_1^2=1$, $f_0^2=f_\infty^2=0$.
We have the linear equivalence relations $f_0 \sim f_\infty$ and 
$C_1 \sim C_{-1}+f_0$.

We take $D_2=C_1$ and $D_1$ 
a smooth rational curve linearly equivalent to $C_{-1}+f_0+f_\infty
\sim C_{-1}+2f_0$. We have $D_2^2=1$
and $D_1^2=3$.

It follows from the above discussion of 
$\F_1(1,3)$ that 
$\tilde{Z}$ is obtained from 
$\F_2$ by blowing-up $d_2$ points on 
the divisor dual to the ray of direction $(0,1)$, one point on the divisor dual to the ray of direction $(-1,0)$ and one point on the divisor dual to the ray of direction $(1,1)$.
As
\[ \langle (0,1),(-1,0)\rangle =1 \,,
\langle (1,1),(-1,0)\rangle =1 \,,
\langle (1,1),(0,1) \rangle =1 \,, \] 
it follows from Section \ref{quiver_construction}
that the corresponding quiver and dimension vectors are given by
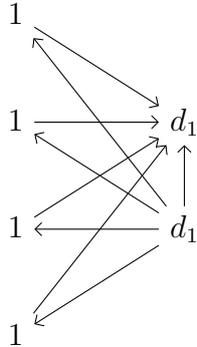
\begin{center}
\begin{tikzpicture}[>=angle 90]
\matrix(a)[matrix of math nodes,
row sep=2em, column sep=2em,
text height=1.5ex, text depth=0.25ex]
{1& &\\
1& &d_1\\
1& &d_1\\
1& &\\};
\path[->](a-1-1) edge  (a-2-3);
\path[->](a-2-1) edge  (a-2-3);
\path[->](a-3-1) edge  (a-2-3);
\path[->](a-4-1) edge  (a-2-3);
\path[<-](a-1-1) edge  (a-3-3);
\path[<-](a-2-1) edge  (a-3-3);
\path[<-](a-3-1) edge  (a-3-3);
\path[<-](a-4-1) edge  (a-3-3);
\path[<-](a-2-3) edge (a-3-3);
\end{tikzpicture}
\end{center}
with $d_2$ vertices of dimension $1$
on the left. This quiver is acyclic and so 
Theorem \ref{main_thm_precise} applies.

\subsection{$\F_2(2,2)$} 
We consider $Y=\F_2$.  We denote
$C_{-2}$, $C_2$, $f_0$, $f_\infty$
the toric divisors of $\F_2$ so that $C_{-2}^2=-2$,
$C_2^2=2$, $f_0^2=f_\infty^2=0$.
We have the linear equivalence relations $f_0 \sim f_\infty$ and 
$C_2 \sim C_{-2}+2f_0$.

We take $D_1=C_2$ and $D_2$ 
a smooth rational curve linearly equivalent to $C_{-2}+f_0+f_\infty
\sim C_2$, intersecting $C_2$ transversally. We have $D_1^2=
D_2^2=2$.


We first explain how to construct a toric model for $(Y,D)$, where $D=D_1 \cup D_2$. We blow-up the two intersection 
points of $D_1$ and $D_2$. 
Let $F_1$ and $F_2$ be the two exceptional divisors.
The strict transforms of the $\PP^1$
fibers (of degree $[f_0]=[f_\infty]$) passing through these points are disjoint 
$(-1)$-curves, which can be contracted. The resulting log Calabi-Yau surface is $\F_0=\PP_1 \times \PP^1$
with its toric boundary.

The fan of $\F_0$ is given by 
the four rays generated by 
$(-1,0)$, $(0,-1)$, $(0,1)$ and $(1,0)$:
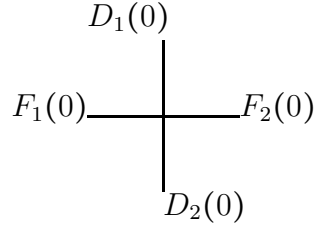
\begin{center}
\setlength{\unitlength}{1cm}
\begin{picture}(10,3)
\thicklines
\put(5,1.5){\line(-1, 0){1}}
\put(5,1.5){\line(0,-1){1}}
\put(5,1.5){\line(0,1){1}}
\put(5,1.5){\line(1,0){1}}
\put(3,1.5){$F_1(0)$}
\put(6,1.5){$F_2(0)$}
\put(4,2.7){$D_1(0)$}
\put(5,0.2){$D_2(0)$}
\end{picture}
\end{center}
We wrote near each ray the corresponding divisor and in parenthesis its self intersection number.

Let $\beta=d_1 [C_{-2}]+d_2 [f_0]
\in H_2(\F_2,\Z)$. We have 
$\beta \cdot D=2d_2$ and 
$\beta \cdot D_1=\beta \cdot D_2=d_2$.
Following the class 
$\beta$ through the previous blow-ups and blow-downs, we get the class 
$\beta_{\F_0} \in H_2(\F_0,\Z)$ such that 
$\beta_{\F_0}\cdot D_1=\beta_{\F_0}\cdot D_2=d_2$
and $\beta_{\F_0}\cdot F_1=\beta_{\F_0}\cdot F_2=
\beta \cdot f_0=d_1$.
Remark that the balancing condition
\[d_2(0,1)+d_2(0,-1)+d_1(1,0)+d_1(-1,0)=0\]
is indeed satisfied.

The surface
$\tilde{Z}$ is obtained from 
$\F_0$ by blowing-up $d_2$ points on 
the divisor dual to the ray of direction $(0,-1)$, one point on the divisor dual to the ray of direction $(-1,0)$ and one point on the divisor dual to the ray of direction $(1,0)$.
As
\[ \langle (-1,0),(0,-1)\rangle =1 \,,
\langle (0,-1),(1,0) \rangle =1 \,, \] 
it follows from Section \ref{quiver_construction}
that the corresponding quiver and dimension vectors are given by
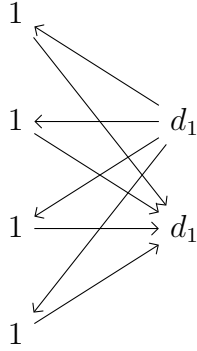
\begin{center}
\begin{tikzpicture}[>=angle 90]
\matrix(a)[matrix of math nodes,
row sep=2em, column sep=2em,
text height=1.5ex, text depth=0.25ex]
{1& &\\
1& &d_1\\
1& &d_1\\
1& &\\};
\path[<-](a-1-1) edge  (a-2-3);
\path[<-](a-2-1) edge  (a-2-3);
\path[<-](a-3-1) edge  (a-2-3);
\path[<-](a-4-1) edge  (a-2-3);
\path[->](a-1-1) edge  (a-3-3);
\path[->](a-2-1) edge  (a-3-3);
\path[->](a-3-1) edge  (a-3-3);
\path[->](a-4-1) edge  (a-3-3);
\end{tikzpicture}
\end{center}
with $d_2$ vertices of dimension $1$ on the left. This quiver is acyclic and so Theorem 
\ref{main_thm_precise} applies.

\subsection{$\F_N(-N,N+4)$} 
We consider $Y=\F_N$.  We denote
$C_{-N}$, $C_N$, $f_0$, $f_\infty$
the toric divisors of $\F_N$ so that $C_{-N}^2=-N$,
$C_N^2=N$, $f_0^2=f_\infty^2=0$.
We have the linear equivalence relations $f_0 \sim f_\infty$ and 
$C_N \sim C_{-N}+Nf_0$.

We take $D_1=C_{-N}$ and $D_2$ 
a smooth rational curve linearly equivalent to $C_N+f_0+f_\infty
\sim C_N+2f_0$, intersecting $C_{-N}$ transversally. We have $D_1^2=-N$ and
$D_2^2=N+4$.


We first explain how to construct a toric model for $(Y,D)$, where $D=D_1 \cup D_2$.
We blow-up the two intersection points of 
$D_1$ and $D_2$. Let 
$F_1$ and $F_2$ be the two exceptional divisors.
The strict transforms of the $\PP^1$ fibers 
(of degree $[f_0]=[f_\infty]$) passing through these points are disjoint interior 
$(-1)$-curves, which can be contracted. The resulting log Calabi-Yau surface is 
$\F_{N+2}$ with its toric boundary.

The fan of $\F_{N+2}$ is given by the four rays generated by 
$(-1,0)$, $(0,-1)$, $(0,1)$ and $(1,N+2)$. 

Let $\beta=d_1 [C_{-N}]+d_2 [f_0]
\in H_2(\F_N,\Z)$. We have 
$\beta \cdot D=-Nd_1+2(d_1+d_2)$,
$\beta \cdot D_1=-Nd_1+d_2$ and $\beta \cdot D_2=2d_1+d_2$.
Following the class $\beta$ through the previous blow-ups and blow-downs, 
we get the class $\beta_{\F_{N+2}}$ such that 
$\beta_{\F_{N+2}}\cdot D_1=-Nd_1+d_2$,
$\beta_{\F_{N+2}}\cdot D_2=2d_1+d_2$,
$\beta_{\F_{N+2}}\cdot F_1=\beta_{\F_{N+2}}\cdot F_2
=\beta\cdot f_0=d_1$.
Remark that the balancing condition
\[(-Nd_1+d_2)(0,1)+(2d_1+d_2)(0,-1)+d_1(-1,0)+d_1(1,N+2)=0\]
is indeed satisfied.

The surface $\tilde{Z}$ is obtained from
$\F_{N+2}$ by blowing-up $2d_1+d_2$ points on the divisor dual to the ray of direction 
$(0,-1)$, one point on the divisor dual to the ray of direction $(-1,0)$ and one point on the divisor dual to the ray of direction 
$(1,N+2)$.
As 
\[\langle (1,N+2),(-1,0)\rangle =N+2
\,,
\langle (-1,0),(0,-1) \rangle =1 \,,
\langle (0,-1),(1,N+2) \rangle =1 \,,\]
it follows from Section
\ref{quiver_construction} that the corresponding quiver and dimension vectors are given by
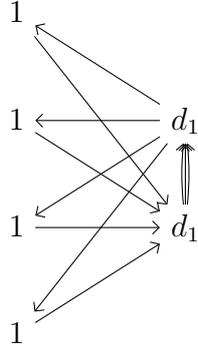
\begin{center}
\begin{tikzpicture}[>=angle 90]
\matrix(a)[matrix of math nodes,
row sep=2em, column sep=2em,
text height=1.5ex, text depth=0.25ex]
{1& &\\
1& &d_1\\
1& &d_1\\
1& &\\};
\path[<-](a-1-1) edge  (a-2-3);
\path[<-](a-2-1) edge  (a-2-3);
\path[<-](a-3-1) edge  (a-2-3);
\path[<-](a-4-1) edge  (a-2-3);
\path[->](a-1-1) edge  (a-3-3);
\path[->](a-2-1) edge  (a-3-3);
\path[->](a-3-1) edge  (a-3-3);
\path[->](a-4-1) edge  (a-3-3);
\path[->](a-3-3) edge  (a-2-3);
\path[->][bend left=5](a-3-3) edge  (a-2-3);
\path[->][bend right=5](a-3-3) edge  (a-2-3);
\end{tikzpicture}
\end{center}
with $2d_1+d_2$ vertices of dimension $1$ on the left and 
$N+2$ vertical arrows between the two vertices on the right.
Remark that this quiver contains oriented cycles and so Theorem 
\ref{main_thm_precise} does not apply.

\subsection{$\F_N(N+4,-N)$}
We consider $Y=\F_N$.  We denote
$C_{-N}$,$C_N$, $f_0$, $f_\infty$
the toric divisors of $\F_N$ so that $C_{-N}^2=-N$,
$C_N^2=N$, $f_0^2=f_\infty^2=0$.
We have the linear equivalence relations $f_0 \sim f_\infty$ and 
$C_N \sim C_{-N}+Nf_0$.

We take $D_2=C_{-N}$ and $D_1$ 
a smooth rational curve linearly equivalent to $C_N+f_0+f_\infty
\sim C_N+2f_0$, intersecting $C_{-N}$ transversally. We have $D_2^2=-N$ and
$D_1^2=N+4$.

It follows from the above discussion of 
$\F_N(-N,N+4)$ that the surface 
$\tilde{Z}$ is obtained from
$\F_{N+2}$ by blowing-up
$-Nd_1+d_2$ points on the divisor dual to the ray of direction $(0,1)$, 
one point on the divisor dual to the ray of direction $(-1,0)$ and one point on  the divisor dual to the ray of direction 
$(1,N+2)$.
As 
\[\langle (1,N+2),(-1,0)\rangle =N+2
\,,
\langle (0,1),(-1,0) \rangle =1 \,,
\langle (1,N+2),(0,1) \rangle =1 \,,\]
it follows from Section
\ref{quiver_construction} that the corresponding quiver and dimension vectors are given by
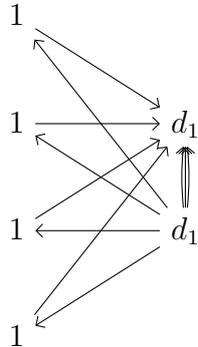
\begin{center}
\begin{tikzpicture}[>=angle 90]
\matrix(a)[matrix of math nodes,
row sep=2em, column sep=2em,
text height=1.5ex, text depth=0.25ex]
{1& &\\
1& &d_1\\
1& &d_1\\
1& &\\};
\path[->](a-1-1) edge  (a-2-3);
\path[->](a-2-1) edge  (a-2-3);
\path[->](a-3-1) edge  (a-2-3);
\path[->](a-4-1) edge  (a-2-3);
\path[<-](a-1-1) edge  (a-3-3);
\path[<-](a-2-1) edge  (a-3-3);
\path[<-](a-3-1) edge  (a-3-3);
\path[<-](a-4-1) edge  (a-3-3);
\path[<-](a-2-3) edge  (a-3-3);
\path[<-][bend left=5](a-2-3) edge  (a-3-3);
\path[<-][bend right=5](a-2-3) edge  (a-3-3);
\end{tikzpicture}
\end{center}
with $-Nd_1+d_2$ vertices of dimension $1$ on the left and  $N+2$
vertical arrows between the two vertices on the right. This quiver is acyclic and so Theorem \ref{main_thm_precise} applies.

\newcommand{\etalchar}[1]{$^{#1}$}

\vspace{+8 pt}
\noindent
Institute for Theoretical Studies \\
ETH Zurich \\
8092 Zurich, Switzerland \\
pboussea@ethz.ch

\end{document}